\documentclass[preprint]{article}

\usepackage{amsmath, amsthm, amsfonts,amssymb}
\usepackage{bm}
\usepackage[left=1in,right=1in,top=1.3in,bottom=1.3in]{geometry}
\usepackage{fancyhdr}
\usepackage[affil-it]{authblk}
\usepackage{graphicx, color}
\usepackage{subfigure}
\usepackage{epstopdf}
\usepackage{float}
\usepackage{hyperref}

\makeatletter
\newcommand{\subjclass}[2][AMS 2000]{%
  \let\@oldtitle\@title%
  \gdef\@title{\@oldtitle\footnotetext{#1 \emph{subject classification.} #2}}%
}
\newcommand{\keywords}[1]{%
  \let\@@oldtitle\@title%
  \gdef\@title{\@@oldtitle\footnotetext{\emph{Key words and phrases.} #1.}}%
}
\makeatother  

\theoremstyle{plain}
\newtheorem{thm}{Theorem}[section]
\newtheorem{cor}[thm]{Corollary}
\newtheorem{lem}[thm]{Lemma}
\newtheorem{assu}[thm]{Assumption}

\newtheorem{defn}[thm]{Definition}

\theoremstyle{remark}
\newtheorem{rem}[thm]{Remark}

\numberwithin{equation}{section}

\date{}
\begin{document}
\title{Singular vector distribution of sample covariance matrices}


\author{Xiucai Ding }
\affil{Department of Statistical Sciences, University of Toronto}

\subjclass{15B52, 15A18}
\keywords{Sample covariance matrices, Singular vector distribution, Universality, Deformed Marcenko-Pastur law}

\maketitle

\begin{abstract} 
We consider a class of sample covariance matrices of the form $Q=TXX^{*}T^*,$ where $X=(x_{ij})$ is an $M \times N$ rectangular matrix consisting of i.i.d entries and $T$ is a deterministic matrix satisfying $T^*T$ is diagonal.  Assuming $M$ is comparable to $N$, we prove that the distribution of the components of the singular vectors  close to the edge singular values agrees with that of Gaussian ensembles provided the first two moments of $x_{ij}$ coincide with the Gaussian random variables. For the singular vectors associated with the bulk singular values, the same conclusion holds if the first four moments of $x_{ij}$ match with those of Gaussian random variables. Similar results have been proved for Wigner matrices by Knowles and Yin in \cite{KY}. 
\end{abstract}

\section{Introduction}
In the analysis of multivariate data, a large collection of statistical methods including principal component analysis, regression analysis and clustering analysis require the knowledge of covariance matrices \cite{CRZ}. The advance of data acquisition and storage has led to datasets for which the sample size $N$ and the number of variables $M$ are both large. This high dimensionality cannot be handled using
the classical statistical theory.

For applications involving large dimensional covariance matrices, it is important to understand the local behavior of the 
the singular values and  vectors. Assuming that $M$ is comparable to $N,$  the spectral analysis of the singular values has attracted considerable interests since the seminal work of Marcenko and Pastur \cite{MP}. Since then, numerous researchers have contributed to weakening the conditions on matrix entries  as well as extending the class of matrices for which the empirical spectral distributions (ESD) have nonrandom limits. For a detailed review, we refer to the monograph \cite{BS}. Besides the ESD of the singular values, the limiting distributions of the extreme singular values were analysed in a collection of celebrated papers. The results were first proved for Wishart matrix  (i.e sample covariance matrices obtained from a data matrix consisting of i.i.d centered real or complex Gaussian entries) in \cite{IJ2,TW}; later on they were proved for matrices with entries satisfying arbitrary sub-exponential distribution \cite{BPZ1, PY, PY1}.  And most recently, the weakest moment condition  was given in \cite{DY}. 

However, less is known for the singular vectors. Therefore, recent research on the limiting behaviour
of the singular vectors has attracted considerable interests among mathematicians and statisticians. Silverstein firstly derived the limit theorems of the eigenvectors of covariance matrices \cite{JS2}; later  on the results were proved for a general class of covariance matrices \cite{BMP}. The delocalization property for the eigenvectors were shown in \cite{BKYY,PY1}. And the universal properties of the eigenvectors of covariance matrices were analysed in \cite{BEKYY,BKYY,LP, TV}. For a recent survey of the results, we refer to \cite{ORVW}. In this paper, we prove the universality for the distribution of the singular vectors for a general class of covariance matrices of the form $Q=T XX^{*} T^*$, where $T$ is deterministic matrix satisfying $T^*T$ is diagonal.

The covariance matrix $Q$ contains a general class of covariance structures and random matrix models \cite[Section 1.2]{BKYY}. The  singular values analysis of $Q$
has attracted considerable attention, see among others, the limiting spectral distribution and Stieltjes transform were derived in \cite{JS1}, the Tracy-Widom asymptotics of the extreme eigenvalues were proved in \cite{BPZ1, KN, KY1, LS4} and the anisotropic local law was proposed in \cite{KY1}. It is notable that in general, $Q$ contains the spiked covariance matrices  \cite{BBP, BGGM, BGN, BKYY, IJ2}. In such models, the ESD of $Q$ still satisfies the Marcenko-Pastur (MP) law and some of the eigenvalues of $Q$ will detach from the bulk and become outliers. However, in this paper, we adapt the regularity Assumption \ref{defnregularity} to rule out the outliers for the purpose of universality discussion.  Actually,  it is shown in \cite{CDF,KY2} that, the distributions of the outliers are not universal.  


In this paper, we study the singular vector distribution of  $Q$.  We prove the universality for the components of the edge singular vectors by assuming the matching of the first two moments of the matrix entries. We also prove similar results in the bulk, under stronger assumption that the first four moments of the two ensembles match. Similar results have been proved for Wigner matrices in \cite{KY}. 
\paragraph{ 1.1. Sample covariance matrices with a general class of populations.} We first introduce some notations. Throughout the paper, we will use
\begin{equation}\label{def_d_n}
r=\lim_{N \rightarrow \infty}r_N=\lim_{N \rightarrow  \infty } \frac{N}{M}.
\end{equation}
Let $X=(x_{ij})$ be an $M\times N$ data matrix with  centered entries $x_{ij}= N^{-1/2}q_{ij}$, $1 \leq i \leq M$ and $1 \leq j \leq N ,$ where $q_{ij}$ are i.i.d random variables with unit variance and for all  $p \in \mathbb{N}$, there exists a constant $C_p$, such that $q_{11}$ satisfies the following condition
\begin{equation}
\mathbb{E}\vert q_{11} \vert^p \leq C_p . \label{INTRODUCTIONEQ}
\end{equation}

We consider the sample covariance matrix $Q=T XX^* T^*,$ where $T$ is a deterministic matrix satisfying $T^*T$ is a positive diagonal matrix. Using the QR factorization \cite[Theorem 5.2.1]{GL}, we find that $T=U\Sigma^{1/2},$ where $U$ is an orthogonal matrix  and $\Sigma$ is a positive diagonal matrix.  Denote $Y=\Sigma^{1/2} X,$ and the singular value decomposition of  $Y$ as $Y = \sum\limits_{k = 1}^{N\wedge M} {\sqrt {\lambda_k } \xi_k } \zeta _{k}^* $,  where $\lambda_k, k=1,2,\cdots, N \wedge M$ are the nontrivial eigenvalues of $Q$ and $\{\xi_{k}\}_{k=1}^{M}$ and $\{\zeta_{k}\}_{k=1}^{N}$ are orthonormal bases of $\mathbb R^{M}$ and $\mathbb R^{N}$ respectively. First of all, we observe that 
\begin{equation*}
X^*T^*TX=Y^*Y=\mathbf{Z} \Lambda_N \mathbf{Z}^*,
\end{equation*} 
where the columns of $\mathbf{Z}$  are $\zeta_1, \cdots, \zeta_N$ and $\Lambda_N$ is a diagonal matrix with entries $\lambda_1, \cdots, \lambda_N.$ As a consequence, $U$ will not influence the right singular vectors of $Y$.  Next, we have
\begin{equation*}
TXX^*T^*=UYY^* U^*=U\bm{\Xi}\Lambda_M \bm{\Xi}^* U^*,
\end{equation*}
where  the columns of $\bm{\Xi}$ are $\xi_k, k=1,2,\cdots, M$ and $\Lambda_M$ is a diagonal matrix containing $\lambda_1, \cdots, \lambda_M.$  Use the fact that the product of orthogonal matrices is again orthogonal, we conclude that the left singular vectors  of $TX$ are  $\hat{\xi}_k: =U\xi_k.$ Hence, the components of $\hat{\xi}_k$ is a linear combination of $\xi_k$.  For instance, we have
\begin{equation*}
\hat{\xi}_k(i) \hat{\xi}_k(j)=\sum_{p_1=1}^M \sum_{p_2=1}^M U_{i p_1} U_{j p_2} \xi_{k}(p_1) \xi_{k}(p_2).
\end{equation*}
By the delocalization result (see Lemma \ref{de_local}) and dominated convergence theorem, we only need to consider the universality of the entries of $\xi_k.$ The above discussion shows that, we can make the following assumptions on $T:$
\begin{equation} \label{defn_sigma_values}
T \equiv \Sigma^{1/2}=\operatorname{diag}\{\sigma_1^{1/2}, \cdots, \sigma_M^{1/2}\}, \ \text{with} \  \sigma_1 \geq \sigma_2 \geq \cdots \geq \sigma_M  >0.
\end{equation} 
We denote the empirical spectral distribution  of $\Sigma$ by
\begin{equation} \label{def_pi}
\pi:=\frac{1}{M} \sum_{i=1}^M  \delta_{\sigma_i}.
\end{equation}
Suppose that there exists some small positive constant $\tau$ such that,  
\begin{equation} \label{assm_11}
\tau< \sigma_M \leq \sigma_1 \leq \tau^{-1}, \ \tau \leq r \leq \tau^{-1}, \ \pi([0,\tau]) \leq 1-\tau. 
\end{equation}
 For definiteness, in this paper we focus on the real case, i.e. all the entries $x_{ij}$ are real. However, it is clear that our results and proofs can be applied to the complex case after minor modifications if we assume in addition that $\operatorname{Re}\, x_{ij}$ and $\operatorname{Im}\, x_{ij}$ are independent centered random variables with the same variance. 
To avoid repetition, we summarize the basic assumptions for future reference. 
\begin{assu}\label{assu_main} We assume $X$ is an $M \times N $ matrix with centered i.i.d entries satisfying (\ref{def_d_n}) and (\ref{INTRODUCTIONEQ}). We also assume that $T$ is a deterministic $M \times M$ matrix satisfying (\ref{defn_sigma_values}) and (\ref{assm_11}).
\end{assu} 

From now on, we will always use $Y=\Sigma^{1/2}X$ and its singular value decomposition $Y = \sum\limits_{k = 1}^{N\wedge M} {\sqrt {\lambda_k } \xi_k } \zeta _{k}^*$, where $\lambda_1 \geq \lambda_2 \geq \cdots \geq \lambda_{M \wedge N}.$
 
\paragraph{1.2. Deformed Marcenko-Pastur law.} We use this subsection to  discuss  the empirical spectral distribution of  $X^*T^*TX,$ where we basically follow the discussion of \cite[Section 2.2]{KY1}. It is well-known that if $\pi$ is a compactly supported probability measure on $\mathbb{R},$ and let $r_N>0,$ then for any $z \in \mathbb{C}_{+},$ there is a unique $m \equiv m_N(z) \in \mathbb{C}_{+}$ satisfying 
\begin{equation} \label{selfconsistenteuqation1}
\frac{1}{m}=-z+\frac{1}{r_N} \int \frac{x}{1+mx}\pi (dx).
\end{equation} 
\noindent We refer the reader to \cite[Lemma 2.2]{KY1} and \cite[Section 5]{SC} for more detail. In this paper, we define the deterministic function $m \equiv m(z)$ as the unique solution of (\ref{selfconsistenteuqation1}) with $\pi$ defined in (\ref{def_pi}). We define by $\rho$ the probability measure associated with $m$ (i.e. $m$ is the Stieltjes  transform of $\rho$) and call it the asymptotic density of $X^*T^*TX$.  Our assumption (\ref{assm_11}) implies that the spectrum of $\Sigma$ cannot be concentrated at zero, thus it ensures $\pi$ is a compactly supported probability measure. Therefore, $m$ and $\rho$ are well-defined.

Let $z \in \mathbb{C}_{+},$ then $m \equiv m(z)$ can be characterized as the unique solution of the equation
\begin{equation}  \label{defnf}
z=f(m), \ \operatorname{Im} m \geq 0, \ \ \text{where}  \ f(x):=-\frac{1}{x}+\frac{1}{r_N} \sum_{i=1}^M \frac{\pi(\{\sigma_i\})}{x+\sigma_i^{-1}}.
\end{equation}
The behaviour of $\rho$ can be entirely understood by the analysis of $f$. We summarize the elementary properties of $\rho$ as the following lemma. It can be found in \cite[Lemma 2.4, 2.5 and 2.6]{KY1}. 
\begin{lem} \label{propertiesofrho} Denote $\overline{\mathbb{R}}=\mathbb{R} \cup \{\infty \}$, then f defined in (\ref{defnf}) is smooth on the $M+1$ open intervals of  $ \overline{\mathbb{R}}$ defined through
\begin{equation*}
I_1:=(-\sigma_1^{-1},0), \ I_i:=(-\sigma_i^{-1}, -\sigma_{i-1}^{-1}), \ i=2,\cdots, M, \ I_0:= \overline{R}/ \cup_{i=1}^M \bar{I}_i.
\end{equation*}
We also introduce a multiset $\mathcal{C} \subset \overline{\mathbb{R}}$ containing the critical points of $f$, using the conventions that a nondegenerate critical point is counted once and a degenerate critical point will be counted twice. In the case $r_N=1,$  $\infty$ is a nondegenerate critical point. With the above notations, we have
\begin{itemize}
\item({\bf Critical Points})  $: \vert \mathcal{C} \cap I_0 \vert=\vert \mathcal{C} \cap I_1 \vert=1$ and $\vert \mathcal{C} \cap I_i \vert \in \{0,2\}$ for $i=2, \cdots, M.$ Therefore, $\vert \mathcal{C} \vert=2p$, where for convenience, we denote by $x_1 \geq x_2 \geq \cdots \geq x_{2p-1}$ be the $2p-1$ critical points in $I_1 \cup \cdots \cup I_M$ and $x_{2p}$ be the unique critical point in $I_0$.
\item ({\bf Ordering}) : Denote $a_k:=f(x_k) $, we have $a_1 \geq \cdots \geq a_{2p}.$ Moreover, we have $x_k=m(a_k)$ by assuming $m(0):= \infty$ for $r_N=1$. Furthermore, for $k=1,\cdots,2p,$ there exists a constant $C$ such that $0\leq a_{k} \leq C$. 
\item ({\bf Structure of $\rho$}): $\textsl{supp} \ \rho \cap (0, \infty)=(\cup_{k=1}^p[a_{2k}, a_{2k-1}])\cap (0, \infty)$. 
\end{itemize}
\end{lem}
With the above definitions and properties, we now introduce the key regularity assumption on $\Sigma$.
\begin{assu} \label{defnregularity} Fix  $\tau>0$,  we say that \\
(i) The edges $a_k,\ k=1,\cdots, 2p$ are regular if 
\begin{equation} \label{defnregularityequation1}
a_k \geq \tau, \ \min_{l \neq k} \vert a_k-a_l \vert \geq \tau, \ \min_{i} \vert x_k+\sigma_{i}^{-1} \vert \geq \tau.
\end{equation}
(ii) The bulk components $k=1,\cdots,p$ are regular if for any fixed $\tau^{\prime}>0$ there exists a constant $c\equiv c_{\tau,\tau^{\prime}}$ such that the density of $\rho$ in $[a_{2k}+\tau^{\prime}, a_{2k-1}-\tau^{\prime}]$ is bounded from below by $c$.
\end{assu}
\begin{rem} The second condition in (\ref{defnregularityequation1}) states that the gap in the spectrum of $\rho$ adjacent to $a_k$ can be well separated when $N$ is sufficiently large. And the third condition ensures a square root behaviour of $\rho$ in a small neighbourhood of $a_k$. To be specific,  consider the right edge of the $k$-th bulk component, by (A.12) of \cite{KY1},  there exists some small constant $c>0$, such that $\rho$ has the following square root behavior
\begin{equation} \label{squarerootequation}
\rho(x) \sim \sqrt{a_{2k-1}-x}, \ x \in[a_{2k-1}-c, a_{2k-1}].  
\end{equation} 
\noindent As a consequence, it will rule out the outliers. The bulk regularity imposes a lower bound on the density of eigenvalues away from the edges. For examples of matrices $\Sigma$ verifying the regularity conditions, we refer to \cite[Example 2.8 and 2.9]{KY1}.    
\end{rem}

\paragraph{1.3. Main results.} This subsection is devoted to providing the main results of this paper. We first introduce some notations. Recall that the nontrivial classical eigenvalue locations $\gamma_1 \geq \gamma_2 \geq \cdots \geq \gamma_{M \wedge N}$ of $Q$ are defined as $\int_{\gamma_i}^{\infty} d \rho =\frac{i-\frac{1}{2}}{N}. $  By Lemma \ref{propertiesofrho}, there are $p$ bulk components in the spectrum of $\rho$. For $k=1, \cdots, p$, we define the classical number of eigenvalues of the $k$-th bulk component through $N_k:= N \int_{a_{2k}}^{a_{2k-1}} d \rho.$ When $p \geq 1$, we relabel $\lambda_i$ and $\gamma_i$ separately for each bulk component $k=1,\cdots,p$ by introducing
\begin{equation} \label{relabel}
\lambda_{k,i}:=\lambda_{i+\sum_{l <k} N_l}, \ \gamma_{k,i}:=\gamma_{i+\sum_{l < k} N_l} \in (a_{2k}, a_{2k-1}).
\end{equation}
Equivalently, we can characterize $\gamma_{k,i}$ through 
\begin{equation} \label{defn_rigidity}
\int_{\gamma_{k,i}}^{a_{2k-1}} d \rho=\frac{i-\frac{1}{2}}{N}. 
\end{equation}
In the present paper, we will use the following assumption for the technical purpose of the application of the anisotropic local law. 
\begin{assu}\label{assum_nozero}
For $k=1,2,\cdots, p, \ i=1,2,\cdots,N_k,$ $\gamma_{k,i} \geq \tau,$ for some constant $\tau>0.$
\end{assu}
We define the index sets $\mathcal I_1:=\{1,...,M\}, \ \ \mathcal I_2:=\{M+1,...,M+N\}, \ \ \mathcal I:=\mathcal I_1\cup\mathcal I_2.$
We will consistently use the latin letters $i,j\in\mathcal I_1$, greek letters $\mu,\nu\in\mathcal I_2$, and $s,t\in\mathcal I$. Then we label the indices of the matrix according to $X=(X_{i \mu}: i \in \mathcal{I}_1, \mu \in \mathcal{I}_2).$ Similarly, we can label the entries of $\xi_k \in  \mathbb{R}^{\mathcal{I}_1}, \zeta_k \in \mathbb{R}^{\mathcal{I}_2}.$ In the $k$-th bulk component, $k=1,2,\cdots, p,$  we rewrite the index of $\lambda_{\alpha^{\prime}}$ as
\begin{equation}\label{defnalphaprime1}
\alpha^{\prime}:=l+\sum_{t < k}N_t, \ \text{when} \ \alpha^{\prime}-\sum_{t<k}N_t < \sum_{t \leq k} N_t-\alpha^{\prime},
\end{equation}
\begin{equation} \label{defnalphaprime2}
\alpha^{\prime}:=-l+1+\sum_{t \leq k}N_t, \ \text{when} \ \alpha^{\prime}-\sum_{t<k}N_t > \sum_{t \leq k} N_t-\alpha^{\prime}.
\end{equation}
\noindent In this paper, we will always say that \emph{$l$ is associated with $\alpha^{\prime}.$} Note that $\alpha^{\prime}$ is the index of $\lambda_{k,l}$ before the relabeling of (\ref{relabel}) and the two cases correspond to the right and left edges respectively.  Our main result on the distribution of the components of the singular vectors near the edge is the following theorem.  For any positive integers $m,k,$  some function $\theta: \mathbb{R}^m \rightarrow \mathbb{R}$ and $x=(x_1, \cdots, x_m) \in \mathbb{R}^m, $ we denote 
\begin{equation} \label{defn_thetadifferential}
\partial^{(k)} \theta(x)=\frac{\partial^k \theta(x)}{\partial x_1^{k_1} \partial x_2^{k_2}\cdots \partial x_{m}^{k_{m}}}, \ \sum_{i=1}^{m}k_i=k, \ k_1, k_2, \cdots, k_{m} \geq 0,
\end{equation}
and $||x||_2$ to be its $l_2$ norm. Denote $Q_G:= \Sigma^{1/2}X_G X_G^* \Sigma^{1/2},$ where $X_G$ is GOE and $\Sigma$ satisfies (\ref{defn_sigma_values}) and (\ref{assm_11}).
\begin{thm}[Edge universality in a single bulk component] \label{thm_edge} For $Q_V=\Sigma^{1/2}X_VX_V^* \Sigma^{1/2}$ satisfying Assumption \ref{assu_main}, let  $\mathbb{E}^G, \mathbb{E}^V$ denote the expectations with respect to $X_G, X_V.$ Consider the $k$-th bulk component,  $k=1,2,\cdots,p,$ and $l$ defined in (\ref{defnalphaprime1}) or (\ref{defnalphaprime2}), under Assumption \ref{defnregularity} and \ref{assum_nozero},  for any choices of indices $i, j \in \mathcal{I}_1, \ \mu, \nu \in \mathcal{I}_2$, there exists a $\delta \in (0,1),$ when $l \leq N_k^{\delta},$ we have
\begin{equation*}
\lim_{N \rightarrow \infty} [\mathbb{E}^V-\mathbb{E}^G] \theta( N \xi_{\alpha^{\prime}}(i) \xi_{\alpha^{\prime}}(j),  N \zeta_{\alpha^{\prime}}(\mu) \zeta_{\alpha^{\prime}}(\nu))=0,
\end{equation*}
where $ \theta$ is a smooth function in $\mathbb{R}^2$ that satisfies
\begin{equation} \label{defn_theta}
\vert \partial^{(k)} \theta(x) \vert \leq C (1+ \vert \vert x \vert \vert_2)^C, \ k=1,2,3, \ \text{with some constant} \ C>0 . 
\end{equation}
\end{thm}
\begin{thm}[Edge universality for several bulk components] \label{thm_edge_sev}
 For $Q_V=\Sigma^{1/2}X_VX_V^* \Sigma^{1/2}$ satisfying Assumption \ref{assu_main}. Consider the $k_1$-th, $\cdots$, $k_n$-th bulks, $k_1, \cdots, k_n \in \{1,2,\cdots, p\},   n \leq p$,  for  $l_{k_i}$ defined in (\ref{defnalphaprime1}) or (\ref{defnalphaprime2}) associated with the $k_i$-th bulk component, $ i=1,2,\cdots, n$,  under Assumption \ref{defnregularity} and \ref{assum_nozero}, for any choices of indices $i, j \in \mathcal{I}_1, \ \mu, \nu \in \mathcal{I}_2$,   there exists a $\delta \in (0,1),$ when $ l_{k_i} \leq N_{k_i}^{\delta},$ where $l_{k_i}$ is associated with $\alpha_{k_i}^{\prime}, \ i=1,2,\cdots,n,$  we have
\begin{equation*}
\lim_{N \rightarrow \infty} [\mathbb{E}^V-\mathbb{E}^G] \theta( N \xi_{\alpha_{k_1}^{\prime}}(i) \xi_{\alpha_{k_1}^{\prime}}(j),  N \zeta_{\alpha_{k_1}^{\prime}}(\mu) \zeta_{\alpha_{k_1}^{\prime}}(\nu), \cdots, N \xi_{\alpha_{k_n}^{\prime}}(i) \xi_{\alpha_{k_n}^{\prime}}(j),  N \zeta_{\alpha_{k_n}^{\prime}}(\mu) \zeta_{\alpha_{k_n}^{\prime}}(\nu))=0,
\end{equation*}
where $ \theta$ is a smooth function in $\mathbb{R}^{2n}$ that satisfies
\begin{equation} \label{defn_theta2nd}
\vert \partial^{(k)} \theta(x) \vert \leq C (1+ \vert \vert x \vert \vert_2)^C, \ k=1,2,3,  \ \text{with some constant} \ C>0 . 
\end{equation}
\end{thm}
\begin{rem} The results  in Theorem  \ref{thm_edge} and \ref{thm_edge_sev} can be easily extended to a general form containing more entries of the singular vectors using a general form of Green function comparison argument. For example, to extend Theorem \ref{thm_edge}, we consider the $k$-th bulk component and choose any positive integer $\beta$, under Assumption \ref{defnregularity} and \ref{assum_nozero}, for any choices of indices $i_1, j_1, \cdots, i_{\beta}, j_{\beta} \in \mathcal{I}_1$ and $\mu_1, \nu_1, \cdots, \mu_{\beta}, \nu_{\beta} \in \mathcal{I}_2$,  for the corresponding $l_{i}$ defined in (\ref{defnalphaprime1}) or (\ref{defnalphaprime2}),  $i=1,2,\cdots, \beta$,  there exists some $0<\delta<1$ with $ 0< \max_{1\leq i \leq \beta} \{ l_{i} \} \leq N_k^{\delta}$,  we have
\begin{equation} \label{generealsigunlaredge}
\lim_{N \rightarrow \infty} [\mathbb{E}^V-\mathbb{E}^G] \theta(N  \xi_{\alpha^{\prime}_1}(i_1) \xi_{\alpha_1^{\prime}}(j_1), N  \zeta_{\alpha_1^{\prime}}(\mu_1) \zeta_{\alpha^{\prime}_1}(\nu_1) ,\cdots, N  \xi_{\alpha_{\beta}^{\prime}}(i_{\beta}) \xi_{\alpha^{\prime}_{\beta}}(j_{\beta}), N  \zeta_{\alpha_{\beta}^{\prime}}(\mu_{\beta}) \zeta_{\alpha_{\beta}^{\prime}}(\nu_{\beta}))=0,
\end{equation}
where $\theta \in \mathbb{R}^{2 \beta}$ is a smooth function function satisfying  $\vert \partial^{(k)} \theta(x) \vert \leq C (1+ \vert \vert x \vert \vert_2)^C, \ k=1,2,3,$  with some constant $C>0.$  Similarly, we can extend Theorem \ref{thm_edge_sev} to contain more entries of singular vectors.
\end{rem}

Recall (\ref{relabel}),  denote $\varpi_k:=(\vert f^{\prime \prime}(x_k) \vert/2)^{1/3}, $  $\ k=1,2,\cdots, 2p,$ for any positive integer $h$,  we define
\begin{equation*}
\mathbf{q}_{2k-1,h}:=\frac{N^{\frac{2}{3}}}{\varpi_{2k-1}}( \lambda_{k,h}-a_{2k-1}), \ \mathbf{q}_{2k,h}:=-\frac{N^{\frac{2}{3}}}{\varpi_{2k}}(\lambda_{k,N_k-h+1}-a_{2k}).
\end{equation*} 
Consider a smooth function $\theta \in \mathbb{R}$ whose third derivative $\theta^{(3)}$ satisfying $ |\theta^{(3)}(x)| \leq C(1+|x|)^C$, for some constant $C>0.$ Then by \cite[Theorem 3.18]{KY1}, we have
\begin{equation} \label{edgeeigenvalue}
\lim_{N \rightarrow \infty} [\mathbb{E}^V-\mathbb{E}^G] \theta(\mathbf{q}_{k, h})=0.
\end{equation}
Together with Theorem \ref{thm_edge}, we have the following corollary, which is an analogy of \cite[Theorem 1.6]{KY}.  Denote $t=2k-1$ if $\alpha^{\prime}$ is of (\ref{defnalphaprime1}) and $2k$ if $\alpha^{\prime}$ is of (\ref{defnalphaprime2}).
\begin{cor}[Edge joint distribution in a single bulk] \label{corollaryedge}  Under the assumptions of Theorem \ref{thm_edge}, for some positive integer $h,$ we have
\begin{equation} \label{coroedgeequ}
\lim_{N \rightarrow \infty} [\mathbb{E}^V-\mathbb{E}^G] \theta(\mathbf{q}_{t, h}, N \xi_{\alpha^{\prime}}(i) \xi_{\alpha^{\prime}}(j), N \zeta_{\alpha^{\prime}}(\mu) \zeta_{\alpha^{\prime}}(\nu))=0,
\end{equation}
where $\theta \in \mathbb{R}^3$ satisfying
\begin{equation} \label{defn_theta3d}
\vert \partial^{(k)} \theta(x) \vert \leq C (1+ \vert \vert x \vert \vert_2)^C, \ k=1,2,3,  \ \text{with some constant} \ C>0 . 
\end{equation}
\end{cor}
Corollary \ref{corollaryedge} can be extended to a general form for several bulk components.  Denote $t_i=2k_i-1$ if $\alpha_{k_i}^{\prime}$ is of (\ref{defnalphaprime1}) and $2k_i$ if $\alpha_{k_i}^{\prime}$ is of (\ref{defnalphaprime2}).
\begin{cor} [Edge joint distribution for several bulks] \label{coroedgesev} Under the assumptions of Theorem \ref{thm_edge_sev}, for some positive integer $h,$  we have 
\begin{equation*}
\lim_{N \rightarrow \infty} [\mathbb{E}^V-\mathbb{E}^G] \theta( \mathbf{q}_{t_1, h},    N \xi_{\alpha_{k_1}^{\prime}}(i) \xi_{\alpha_{k_1}^{\prime}}(j),  N \zeta_{\alpha_{k_1}^{\prime}}(\mu) \zeta_{\alpha_{k_1}^{\prime}}(\nu), \cdots, \mathbf{q}_{t_n,h},  N \xi_{\alpha_{k_n}^{\prime}}(i) \xi_{\alpha_{k_n}^{\prime}}(j),  N \zeta_{\alpha_{k_n}^{\prime}}(\mu) \zeta_{\alpha_{k_n}^{\prime}}(\nu))=0,
\end{equation*}
where $\theta \in \mathbb{R}^{3 n}$ is a smooth function function satisfying 
\begin{equation} \label{defn_theta3nd}
\vert \partial^{(k)} \theta(x) \vert \leq C (1+ \vert \vert x \vert \vert_2)^C, \ k=1,2,3,  \ \text{with some arbitrary} \ C>0 . 
\end{equation}
\end{cor}
\begin{rem} \label{remarkedge}
(i). Similar to (\ref{generealsigunlaredge}), the results in Corollary \ref{corollaryedge} and \ref{coroedgesev} can be easily extended to a general form containing more entries of the singular vectors. For example, to extend Corollary \ref{corollaryedge},  we can choose any positive integers $\beta$ and $h_1, \cdots, h_{\beta},$ under Assumption \ref{defnregularity} and \ref{assum_nozero},  for any choices of indices $i_1, j_1, \cdots, i_{\beta}, j_{\beta} \in \mathcal{I}_1$ and $\mu_1, \nu_1, \cdots, \mu_{\beta}, \nu_{\beta} \in \mathcal{I}_2$, for the corresponding $l_{i}$ defined in (\ref{defnalphaprime1}) or (\ref{defnalphaprime2}), $i=1,2,\cdots, \beta$,  there exists some $0<\delta<1$ with $ \max_{1\leq i \leq \beta} \{ l_{i} \} \leq N_k^{\delta}$,  we have
\begin{equation*}
\lim_{N \rightarrow \infty} [\mathbb{E}^V-\mathbb{E}^G] \theta(\mathbf{q}_{t_1, h_1}, N  \xi_{\alpha^{\prime}_1}(i_1) \xi_{\alpha_1^{\prime}}(j_1), \zeta_{\alpha_1^{\prime}}(\mu_1) \zeta_{\alpha^{\prime}_1}(\nu_1), \cdots, \mathbf{q}_{t_{\beta}, h_{\beta}}, N  \xi_{\alpha_{\beta}^{\prime}}(i_{\beta}) \xi_{\alpha^{\prime}_{\beta}}(j_{\beta}),  N \zeta_{\alpha_{\beta}^{\prime}}(\mu_{\beta}) \zeta_{\alpha_{\beta}^{\prime}}(\nu_{\beta}))=0.
\end{equation*}
where the smooth function $\theta \in \mathbb{R}^{3 \beta}$ satisfies $\vert \partial^{(k)} \theta(x) \vert \leq C (1+ \vert \vert x \vert \vert_2)^C,  \ k=1,2,3, $ for some constant $C.$ \\
(ii). Theorem \ref{thm_edge} and \ref{thm_edge_sev}, Corollary \ref{corollaryedge} and \ref{coroedgesev} still hold true for the complex case, where the moment matching condition is replaced by 
\begin{equation} \label{mme}
\mathbb{E}^G\bar{x}_{ij}^l x_{ij}^u=\mathbb{E}^V \bar{x}_{ij}^l x_{ij}^u,  \ 0 \leq l+u \leq 2.
\end{equation}
(iii).  All the above theorems and corollaries are stronger than their counterparts from \cite{KY} because they hold much further into the bulk components. For instance, in the counterpart of Theorem \ref{thm_edge}, which is \cite[Theorem 1.6]{KY}, the universality was established under the assumption that $l \leq (\log N)^{C\log \log N}.$
\end{rem}

In the bulks, similar results hold under the stronger assumption that the first four moments of the matrix entries match with those of Gaussian ensembles.
\begin{thm}[Bulk universality in a single bulk component] \label{thm_bulk}
 For $Q_V=\Sigma^{1/2}X_VX_V^* \Sigma^{1/2}$ satisfying Assumption \ref{assu_main}.  Assuming that the third and fourth moments of $X_V$ agree with those of $X_G$ and  considering the $k$-th bulk component, $k=1,2,\cdots,p$ and  $l$ defined in (\ref{defnalphaprime1}) or (\ref{defnalphaprime2}) , under Assumption \ref{defnregularity} and \ref{assum_nozero}, for  any choices of indices $i, j \in \mathcal{I}_1, \ \mu, \nu \in \mathcal{I}_2$, there exists a small $\delta \in (0,1),$ when $\delta N_k \leq l \leq (1-\delta)N_k,$ we have
\begin{equation*}
\lim_{N \rightarrow \infty} [\mathbb{E}^V-\mathbb{E}^G] \theta(N \xi_{\alpha^{\prime}}(i) \xi_{\alpha^{\prime}}(j),  N \zeta_{\alpha^{\prime}}(\mu) \zeta_{\alpha^{\prime}}(\nu))=0,
\end{equation*}
where $ \theta$ is a smooth function in $\mathbb{R}^2$ that satisfies
\begin{equation} \label{defn_thetabulk}
\vert \partial^{(k)} \theta(x) \vert \leq C (1+ \vert \vert x \vert \vert_2)^C, \ k=1,2,3,4,5,  \ \text{with some constant} \ C>0 . 
\end{equation}
\end{thm}

\begin{thm}[Bulk universality for several bulk components] \label{thm_bulksev}
For $Q_V=\Sigma^{1/2}X_VX_V^* \Sigma^{1/2}$ satisfying Assumption \ref{assu_main}.  Assuming that the third and fourth moments of $X_V$ agree with those of $X_G,$  consider the $k_1$-th, $\cdots$, $k_n$-th bulks, $k_1, \cdots, k_n \in \{1,2,\cdots, p\},   n \leq p$,  for  $l_{k_i}$ defined in (\ref{defnalphaprime1}) or (\ref{defnalphaprime2}) associated with the $k_i$-th bulk component, $i=1,2,\cdots, n,$ under Assumption \ref{defnregularity} and \ref{assum_nozero}, for  any choices of indices $i, j \in \mathcal{I}_1, \ \mu, \nu \in \mathcal{I}_2$, there exists a $\delta \in (0,1),$ when $\delta N_{k_i}\leq l_{k_i} \leq (1-\delta) N_{k_i}, \ i=1,2,\cdots,n,$ we have
\begin{equation*}
\lim_{N \rightarrow \infty} [\mathbb{E}^V-\mathbb{E}^G] \theta( N \xi_{\alpha_{k_1}^{\prime}}(i) \xi_{\alpha_{k_1}^{\prime}}(j),  N \zeta_{\alpha_{k_1}^{\prime}}(\mu) \zeta_{\alpha_{k_1}^{\prime}}(\nu), \cdots, N \xi_{\alpha_{k_n}^{\prime}}(i) \xi_{\alpha_{k_n}^{\prime}}(j),  N \zeta_{\alpha_{k_n}^{\prime}}(\mu) \zeta_{\alpha_{k_n}^{\prime}}(\nu))=0,
\end{equation*}
where $ \theta$ is a smooth function in $\mathbb{R}^{2n}$ that satisfies
\begin{equation} \label{defn_thetabulk2nd}
\vert \partial^{(k)} \theta(x) \vert \leq C (1+ \vert \vert x \vert \vert_2)^C, \ k=1,2,3, 4,5,  \ \text{with some constant} \ C>0 . 
\end{equation}
\end{thm}

\begin{rem} (i). Similar to Corollary \ref{corollaryedge}, \ref{coroedgesev} and (i) of  Remark \ref{remarkedge},  we can extend the results to the joint distribution containing singular values. We take the extension of Theorem \ref{thm_bulk} as an example.  By (ii) of Assumption \ref{defnregularity}, in the bulk, we have  $\int_{\lambda_{\alpha^{\prime}}}^{\gamma_{\alpha^{\prime}}} d \rho=\frac{1}{N}+o(N^{-1}). $
Using a similar Dyson Brownian motion argument as in \cite{PY1}, combining with Theorem \ref{thm_bulk}, we have
\begin{equation} \label{coroedgeequ}
\lim_{N \rightarrow \infty} [\mathbb{E}^V-\mathbb{E}^G] \theta(\mathbf{p}_{\alpha^{\prime}}, N \xi_{\alpha^{\prime}}(i) \xi_{\alpha^{\prime}}(j), N \zeta_{\alpha^{\prime}}(\mu) \zeta_{\alpha^{\prime}}(\nu))=0.
\end{equation}
where $\mathbf{p}_{\alpha^{\prime}}$ is defined as
\begin{equation*}
\mathbf{p}_{\alpha^{\prime}}:=\rho(\gamma_{\alpha^{\prime}}) N (\lambda_{\alpha^{\prime}}-\gamma_{\alpha^{\prime}}),
\end{equation*}
and  $\theta \in \mathbb{R}^3$ satisfying
\begin{equation*} 
\vert \partial^{(k)} \theta(x) \vert \leq C (1+ \vert \vert x \vert \vert_2)^C,  \ k=1,2,3,4,5,  \ \text{with some constant} \ C>0 . 
\end{equation*}
(ii). Theorem \ref{thm_bulk} and \ref{thm_bulksev} still hold true for the complex case, where the moment matching condition is replaced by 
\begin{equation} \label{mmb}
\mathbb{E}^G\bar{x}_{ij}^l x_{ij}^u=\mathbb{E}^V \bar{x}_{ij}^l x_{ij}^u,  \ 0 \leq l+u \leq 4.
\end{equation}
\end{rem}

\paragraph{1.4.  Applications to statistics.}
In this subsection, we give a few remarks on the possible applications to statistics. It is notable that, in general, the distribution of the singular vectors of sample covariance matrix $Q=TXX^*T^*$ is unknown, even for the GOE case. However, when $T$ is a scalar matrix (i.e $T =c I$, $c>0$), Bourgade and Yau \cite[Appendix C]{PBY} have shown that the entries of the singular vectors are asymptotically normally distributed.  Hence, our universality results imply that under Assumption \ref{assu_main}, \ref{defnregularity} and \ref{assum_nozero}, when $T$ is conformal  (i.e $T^*T =c I$, $c>0$), the entries of the right singular vectors are asymptotically normally distributed. Therefore, this can be used to test the null hypothesis 
\begin{equation}\label{conformal_test}
\mathbf{H}_0: T \ \text{is a conformal matrix} . 
\end{equation} 
The statistical testing problem (\ref{conformal_test}) contains a rich class of hypothesis tests. For instance, when $T=I,$ it reduces to the sphericity test and when $c=1,$ it reduces to test whether the covariance matrix of $X$ is orthogonal \cite{YZB}. 

To illustrate how our results can be used to test (\ref{conformal_test}), we take the example by assuming $c=1$ in the following discussion. Under $\mathbf{H}_0,$ denote the QR factorization of $T$ to be $T=UI,$ the right singular vector of $TX$ is the same of $X,$ $\zeta_k, k=1,2, \cdots, N.$ Using \cite[Corollary 1.3]{PBY}, we find that for $i, k=1,2,\cdots, N,$
\begin{equation}\label{right_check}
\sqrt{N} \zeta_k(i) \rightarrow \mathcal{N},
\end{equation}
where $\mathcal{N}$ is a standard Gaussian random variable. In detail, we can take the following steps to test whether $\mathbf{H}_0$ holds true: \\

1. Randomly choose two index sets $ R_1, R_2 \subset \{1,2,\cdots, N \}$ with $|R_i|=O(1), i=1,2.$

2. Use  Bootstrapping method to sample the  columns of  $Q$ and get a sequence of $M \times N$ matrices $Q_j, j=1,2,\cdots, K.$

3. Select  $\zeta^j_k(i), k \in R_1, i \in R_2$ from $Q_j, j=1,2,\cdots, K.$ Use the classic normality test, for instance the Shapiro-Wilk test  to check whether  (\ref{right_check}) hold true for all the above samples.  Record the number of samples cannot be rejected by the normality test by $A$.

4. Given some pre-chosen significant level $\alpha,$  reject $\mathbf{H}_0$ if $\frac{A}{|R_1||R_2|}< 1-\alpha.$ \\

The other important information from our result is that the singular vectors are completely delocalized. In the low rank matrix denoising problem  \cite{DXC20171},
\begin{equation*}
\hat{S}=TX+S,
\end{equation*} 
where $S$ is a deterministic low rank matrix. Consider the rank one case and assume  the left singular vector $u$ of $S$ is sparse, using the completely delocalization result, it can be shown that  the first left singular vector of $\hat{S}$ has the same sparse structure as that of $u.$ Thus, to estimate the singular vectors of $S,$ we only need to do singular value decomposition on a block matrix of $\hat{S}.$  For more detail, we refer to \cite[Section 2.1]{DXC20171}.  

\vspace{0.4cm}

This paper is organized as follows. In Section \ref{main_result}, we introduce some notations and tools that will be used in our proofs. In Section \ref{edge}, we prove the singular vector distribution near the edge. In Sections \ref{bulk}, we prove the distribution within the bulks. In particular, the Green function comparison arguments are mainly discussed in Section \ref{greenfunction_comp_edge} and Lemma \ref{comparisonbulk}. In the appendix, we prove Lemma \ref{lem_deltaestimatetileeta}. \\

\noindent{\bf Conventions.} We will always use $C$ to denote a generic large positive constant, whose value may change from one line to the next. Similarly, we use $\epsilon$  to denote a generic small positive constant. For two quantities $a_N$ and $b_N$ depending on $N$, the notation $a_N = O(b_N)$ means that $|a_N| \le C|b_N|$ for some positive constant $C>0$, and $a_N=o(b_N)$ means that $|a_N| \le c_N |b_N|$ for some positive constants $c_N\to 0$ as $N\to \infty$. We also use the notation $a_N \sim b_N$ if $a_N = O(b_N)$ and $b_N = O(a_N)$. We write the identity matrix $I_{n\times n}$ as $1$ or $I$ when there is no confusion about the dimension.

\section{Notations and tools} \label{main_result} 
In this section, we introduce some notations and tools which will be used in this paper. Throughout the paper, we will always use $\epsilon_1$ for a small constant and $D_1$ a large constant. Recall that the ESD of an $n \times n$ symmetric matrix $H$ is defined as
\begin{equation*}
F^{(n)}_H(\lambda):=\frac{1}{n} \sum_{i=1}^n \mathbf{1}_{\{\lambda_i(H) \leq \lambda\}}.
\end{equation*}
\noindent For some small constant $\tau>0,$ we define the typical domain for $z=E+i\eta$ as
\begin{equation} 
\mathbf{D}(\tau)= \{ z \in \mathbb{C}_{+}: | E | \leq \tau^{-1}, \ N^{-1+\tau} \leq \eta \leq \tau^{-1}   \} \label{DOMAIN1}.
\end{equation}

\begin{defn}[Stieltjes transform]
Recall that the Green functions for $YY^{*}$ and $Y^{*} Y$ are defined
\begin{equation}\label{def_green}
\mathcal G_1(z):=(YY^{*}-z)^{-1} , \ \ \ \mathcal G_2 (z):=(Y^{*} Y-z)^{-1} , \ \ \ z=E+i\eta \in \mathbb{C}_{+}.
\end{equation}
The Stieltjes transform of the ESD of $Y^{*}Y$ is given by
\begin{equation}
m_2(z)\equiv m_2^{(N)}(z):=\int \frac{1}{x-z}dF^{(N)}_{Y^*Y}(x)=\frac{1}{N}\sum_{i=1}^N (\mathcal G_2)_{ii}(z)=\frac{1}{N} \mathrm{Tr} \, \mathcal G_2(z). \label{EEEE}
\end{equation}
Similarly, we can also define $m_1(z)\equiv m_1^{(M)}(z):= M^{-1}\mathrm{Tr} \, \mathcal G_1(z)$.
\end{defn}

\noindent It has been shown in \cite{DXC20171, DY,KY1,  XYY} that the  linearizing block matrix  is quite useful in dealing with rectangular  matrices. 
\begin{defn}\label{def_linearHG}
For $z\in \mathbb C_+ $, we define the $(N+M)\times (N+M)$ self-adjoint matrix
 \begin{equation}\label{linearize_block}
   H \equiv H(X, \Sigma): = \left( {\begin{array}{*{20}c}
   { -zI} & z^{1/2}Y  \\
   z^{1/2} Y^{*} & {-zI}  \\
   \end{array}} \right),
 \end{equation}
and
 \begin{equation}\label{eqn_defG}
 G \equiv G (X,z):= H^{-1}.
 \end{equation}
 
\end{defn}

By Schur's complement, it is easy to check that
\begin{equation} \label{green2}
G = \left( {\begin{array}{*{20}c}
   { \mathcal{G}_1(z)} &  z^{-1/2}\mathcal{G}_1(z)Y \\
   z^{-1/2}Y^{*} \mathcal{G}_1(z) & z^{-1}Y^*\mathcal{G}_1(z)Y-z^{-1}I  \\
\end{array}} \right)= \left( {\begin{array}{*{20}c}
   z^{-1}Y \mathcal{G}_2(z)Y^{*}-z^{-1}I & z^{-1/2}Y\mathcal G_2(z)  \\
   {z^{-1/2} \mathcal{G}_2(z)Y^*} & { \mathcal G_2(z) }  \\
\end{array}} \right),
\end{equation}
for $\mathcal G_{1,2}$ defined in (\ref{def_green}).  Thus a control of $G$ yields directly a control of $(YY^*-z)^{-1}$ and $(Y^*Y-z)^{-1}$. Moreover, we have
\begin{equation*}
m_1(z)=\frac{1}{M}\sum_{i\in \mathcal I_1}G_{ii}, \ \ m_2(z)=\frac{1}{N}\sum_{\mu \in \mathcal I_2}G_{\mu\mu}.
\end{equation*}
\noindent Recall that $Y=\sum_{i=1}^{M \wedge N} \sqrt{\lambda_k} \xi_k \zeta_k^*,$ $\xi_k \in \mathbb{R}^{\mathcal{I}_1}, \ \zeta_k \in \mathbb{R}^{\mathcal{I}_2},$  by (\ref{green2}), we have
\begin{equation}\label{main_representation}
G(z) = \sum_{k=1}^{M \wedge N} \frac{1}{\lambda_k-z} \left( {\begin{array}{*{20}c}
   { \xi_k \xi^*_k} &  z^{-1/2} \sqrt{\lambda}_k \xi_k \zeta_k^* \\
   z^{-1/2} \sqrt{\lambda}_k \zeta_k \xi_k^* & \zeta_{k} \zeta_{k}^*  \\
\end{array}} \right).
\end{equation}
Denote
\begin{equation} \label{defn_psi}
\Psi(z):=\sqrt{\frac{\operatorname{Im}m(z)}{N \eta}}+\frac{1}{N \eta}, \  \
 \underline{\Sigma}_o:=\left( {\begin{array}{*{20}c}
   {\Sigma} & 0  \\
   0 & {I}  \\
   \end{array}} \right), \ \  \underline{\Sigma}:=\left( {\begin{array}{*{20}c}
   {z^{-1/2}\Sigma^{1/2}} & 0  \\
   0 & {I}  \\
   \end{array}} \right).
\end{equation}
\begin{defn} For $z \in \mathbb{C}_{+}$, we define the $\mathcal{I} \times \mathcal{I}$ matrix 
\begin{equation} \label{convergentlimit}
\Pi(z):=\left( {\begin{array}{*{20}c}
   { -z^{-1}(1+m(z)\Sigma)^{-1}} & 0  \\
   0 & {m(z)}  \\
   \end{array}} \right).
\end{equation}
We will see later from Lemma \ref{lem_anisotropic} that $G(z)$ converges to $\Pi(z)$ in probability.  
\end{defn}
\begin{rem} 
In \cite[Definition 3.2]{KY1}, the linearizing block matrix is defined as 
 \begin{equation}\label{linearize_block1}
   H_{o}: = \left( {\begin{array}{*{20}c}
   { -\Sigma^{-1}} & X  \\
   X^{*} & {-zI}  \\
   \end{array}} \right).
 \end{equation}
 It is easy to check the following relation between (\ref{linearize_block}) and (\ref{linearize_block1})
 \begin{equation} \label{linearblockrelation}
 H=\left( {\begin{array}{*{20}c}
   { z^{1/2}\Sigma^{1/2}} & 0  \\
   0 & {I}  \\
   \end{array}} \right) H_o \left( {\begin{array}{*{20}c}
   { z^{1/2}\Sigma^{1/2}} & 0  \\
   0 & {I}  \\
   \end{array}} \right).
\end{equation} 
In \cite[Definition 3.3]{KY1}, the deterministic convergent limit of $H_o^{-1}$ is
\begin{equation} \label{convergentlimit1}
 \Pi_o(z)=\left( {\begin{array}{*{20}c}
   { -\Sigma (1+m(z)\Sigma)^{-1}}& 0  \\
   0 & {m(z)}  \\
   \end{array}} \right).
\end{equation} 
Therefore, by (\ref{linearblockrelation}), we can get a similar relation between (\ref{convergentlimit}) and (\ref{convergentlimit1})
\begin{equation}\label{pirelationship}
\Pi(z)=\left( {\begin{array}{*{20}c}
   { z^{-1/2}\Sigma^{-1/2}} & 0  \\
   0 & {I}  \\
   \end{array}} \right) \Pi_o(z) \left( {\begin{array}{*{20}c}
   { z^{-1/2}\Sigma^{-1/2}} & 0  \\
   0 & {I}  \\
   \end{array}} \right).
\end{equation}
\end{rem}

%
\begin{defn} \label{minors} We introduce the notation $X^{(\mathbb{T})}$ to represent the $M \times (N- \vert \mathbb{T} \vert)$ minor of $X$ by deleting the $i$-th, $i \in \mathbb{T}$ columns of $X$. For convenience, $(\{i\})$ will be abbreviated to $(i).$ We will keep the name of indices of $X$ for $X^{(\mathbb{T})},$ that is $X_{ij}^{(\mathbb{T})}=\mathbf{1}(j \notin \mathbb{T})X_{ij}.$  We will denote 
\begin{equation}
Y^{(\mathbb{T})}=\Sigma^{1/2}X^{(\mathbb{T})}, \ \mathcal{G}_{1}^{(\mathbb{T})}=(Y^{(\mathbb{T})}Y^{(\mathbb{T})^*}-zI)^{-1}, \  \mathcal{G}_{2}^{(\mathbb{T})}=(Y^{(\mathbb{T})^{*}}Y^{(\mathbb{T})}-zI)^{-1}.
\end{equation}
Consequently, $m_1^{(\mathbb{T})}(z)=M^{-1}\operatorname{Tr} \mathcal{G}_1^{(\mathbb{T})}(z), \ m_2^{(\mathbb{T})}(z)=N^{-1}\operatorname{Tr} \mathcal{G}_2^{(\mathbb{T})}(z).$

\end{defn}

Our key ingredient is the anisotropic local law derived by Knowles and Yin in \cite{KY1}.
\begin{lem} \label{lem_anisotropic} Fix $\tau>0,$ assume (\ref{def_d_n}),  (\ref{INTRODUCTIONEQ}) and (\ref{assm_11}) hold.   Moreover, suppose that every edge $k=1, \cdots, 2p$ satisfies $a_k \geq \tau$ and every bulk component $k=1,\cdots,p$ is regular in the sense of Assumption \ref{defnregularity}. Then for all $z \in \mathbf{D}(\tau)$ and any unit vectors $\mathbf{u}, \mathbf{v} \in \mathbb{R}^{M+N}$, there exists some small constant $\epsilon_1>0$ and large constant $D_1>0$, when $N$ is large enough, with $1-N^{-D_1}$ probability, we have 
\begin{equation}\label{thm_anisotropic_eq11}
\left|<\mathbf{u},\underline{\Sigma}^{-1}(G(z)-\Pi(z)) \underline{\Sigma}^{-1}\mathbf{v}>\right| \leq N^{\epsilon_1} \Psi(z),
\end{equation}
and 
\begin{equation}\label{thm_anisotropic_eq2}
|m_2(z)-m(z)| \leq N^{\epsilon_1} \Psi(z).
\end{equation}
\end{lem}
\begin{proof}
(\ref{thm_anisotropic_eq2}) is already proved in (3.11) of \cite{KY1}. We only need to prove (\ref{thm_anisotropic_eq11}).   By (\ref{linearblockrelation}), we have 
\begin{equation} \label{grelationship}
 G_o(z)=\left( {\begin{array}{*{20}c}
   { z^{1/2}\Sigma^{1/2}} & 0  \\
   0 & {I}  \\
   \end{array}} \right) G(z) \left( {\begin{array}{*{20}c}
   { z^{1/2}\Sigma^{1/2}} & 0  \\
   0 & {I}  \\
   \end{array}} \right).
\end{equation}
By \cite[Theorem 3.6]{KY1}, with $1-N^{-D_1}$ probability, we have 
\begin{equation}\label{thm_anisotropic_eq0}
|<\mathbf{u},\underline{\Sigma}_o^{-1}(G_o(z)-\Pi_o(z)) \underline{\Sigma}_o^{-1}\mathbf{v}>| \leq N^{\epsilon_1} \Psi(z).
\end{equation}
Therefore, by (\ref{pirelationship}), (\ref{grelationship}) and (\ref{thm_anisotropic_eq0}), we conclude our proof.
\end{proof}
 It is easy to derive the following corollary from Lemma \ref{lem_anisotropic}.
\begin{cor}\label{them_specific_bounds} 
Under the assumptions of Lemma \ref{lem_anisotropic}, with $1-N^{-D_1}$ probability, we have 
\begin{equation}\label{thm_anisotropic_eq12}
|<v,(\mathcal{G}_2(z)-m(z))v>|\leq N^{\epsilon_1}\Psi(z), \  |<u,(\mathcal{G}_1(z)+z^{-1}(1+m(z)\Sigma)^{-1})u>| \leq  N^{\epsilon_1}\Psi(z),
\end{equation}
where $v, u$ are unit vectors in $\mathbb{R}^N, \mathbb{R}^M$ respectively. 
\end{cor}

We use the following lemma to characterize the rigidity of eigenvalues within each of the bulk component, which can be found in  \cite[Theorem 3.12]{KY1}.
\begin{lem}\label{lem_rigidity}
Fix $\tau>0$,  assume (\ref{def_d_n}),  (\ref{INTRODUCTIONEQ}) and (\ref{assm_11}) hold. Moreover, suppose that every edge $k=1, \cdots, 2p$ satisfies $a_k \geq \tau$ and every bulk component $k=1,\cdots,p$ is regular in the sense of Assumption \ref{defnregularity}. Recall $N_k$ is the number of eigenvalues within each bulk, then we have that for $ i=1,\cdots, N_k$ satisfying $\gamma_{k,i} \geq \tau$ and $k=1,\cdots, p$, with $1-N^{-D_1}$ probability, we have
\begin{equation} \label{thm_rigidity_equ}
\vert \lambda_{k,i}-\gamma_{k,i} \vert \leq (i \wedge(N_k+1-i))^{-\frac{1}{3}} N^{-\frac{2}{3}+\epsilon_1}.
\end{equation} 
\end{lem}
\noindent Within the bulk, we have stronger result. For small $\tau^{\prime}>0$, denote 
\begin{equation}  \label{domainbulk}
\mathbf{D}_{k}^b:= \{ z \in \mathbf{D}(\tau): E \in [a_{2k}+\tau^{\prime}, a_{2k-1}-\tau^{\prime}] \}, \ k=1,2, \cdots, p,
\end{equation}
as the bulk spectral domain, then \cite[Theorem 3.15]{KY1} gives the following result.
\begin{lem}\label{ridgitybulk} Fix $\tau, \tau^{\prime}>0,$ assume (\ref{def_d_n}), (\ref{INTRODUCTIONEQ}) and (\ref{assm_11}) hold and the bulk component $k=1,\cdots,2p$ is regular in the sense of (ii) of Assumption \ref{defnregularity}. Then for all $i=1,\cdots, N_k$ satisfying $\gamma_{k,i} \in [a_{2k}+\tau^{\prime}, a_{2k-1}-\tau^{\prime}]$, we have (\ref{thm_anisotropic_eq11}) and (\ref{thm_anisotropic_eq2}) hold uniformly for all $z \in \mathbf{D}_{k}^b$  and with $1-N^{-D_1}$ probability, 
\begin{equation} \label{lem_rigity_eq}
|\lambda_{k,i}-\gamma_{k,i}| \leq {N}^{-1+\epsilon_1}.
\end{equation} 
\end{lem}

%
 
 As discussed in \cite[Remark 3.13]{KY1}, Lemma \ref{lem_anisotropic} and  \ref{lem_rigidity} imply the completely delocalization of the singular vectors.
\begin{lem} \label{de_local} Fix $\tau>0,$ under the assumptions of Lemma \ref{lem_anisotropic}, for any $i, \mu$ such that $\gamma_i, \gamma_{\mu} \geq \tau,$ with $1-N^{-D_1}$ probability, we have
\begin{equation} \label{delocalequation}
\max_{i,s_1} \vert \xi_i(s_1) \vert^2+\max_{\mu,s_2} \vert \zeta_{\mu}(s_2) \vert^2 \leq N^{-1+\epsilon_1}.
\end{equation}
\end{lem}
\begin{proof}  By (\ref{thm_anisotropic_eq12}),   with $1-N^{-D_1}$ probability, we have 
$ \max \{ \operatorname{Im} G_{ii}(z), \ \operatorname{Im} G_{\mu \mu}(z) \}=O(1). $
Choose $z_0=E+i\eta_0 $ with $\eta_0 =  N^{-1+\epsilon_1}$ and use the spectral decomposition (\ref{main_representation}), we have
\begin{equation}\label{spectraldecomp}
\sum_{k=1}^{N \wedge M} \frac{ \eta_0}{(E-\lambda_k)^2+\eta_0^2} \vert \xi_{k}(i) \vert^2= \operatorname{Im} G_{ii}(z_0) = O(1),
\end{equation}
\begin{equation}\label{spectraldecomp1}
\sum_{k=1}^{N \wedge M} \frac{ \eta_0}{(E-\lambda_k)^2+\eta_0^2} \vert \zeta_{k}(\mu) \vert^2   = \operatorname{Im} G_{\mu \mu}(z_0) = O(1),
\end{equation}
hold with $1-N^{-D_1}$ probability. Choosing $E=\lambda_k$ in (\ref{spectraldecomp}) and (\ref{spectraldecomp1}), we finish the proof.
\end{proof}

\section{Singular vectors near the edges} \label{edge}
In this section, we prove the universality for the distributions of the edge singular vectors Theorem \ref{thm_edge} and \ref{thm_edge_sev}, as well as the joint distribution between singular values and singular vectors Corollary \ref{corollaryedge} and \ref{coroedgesev}.  The main identities on which we will rely are  
\begin{equation} \label{keykeykeyindentity}
\tilde{G}_{ij}=\sum_{\beta=1}^{M \wedge N} \frac{\eta}{(E-\lambda_{\beta})^2+\eta^2} \xi_{\beta}(i) \xi_{\beta}(j),\ \tilde{G}_{\mu \nu}=\sum_{\beta=1}^{M \wedge N} \frac{\eta}{(E-\lambda_{\beta})^2+\eta^2} \zeta_{\beta}(\mu) \zeta_{\beta}(\nu),
\end{equation}
where $\tilde{G}_{ij}, \tilde{G}_{\mu \nu}$ are defined as 
\begin{equation} \label{defn_g_tilde}
\tilde{G}_{ij}:=\frac{1}{2i}(G_{ij}(z)-G_{ij}(\bar{z})),\  \tilde{G}_{\mu \nu}:=\frac{1}{2i}(G_{\mu \nu}(z)-G_{\mu \nu}(\bar{z})).
\end{equation} 
Due to similarity, we focus our proof on the right singular vectors. The proofs reply on three main steps: (i). Writing $N\zeta_{\beta}(\mu) \zeta_{\beta}(\nu)$ as an integral of $\tilde{G}_{\mu\nu}$ over a random interval with size $O(N^{\epsilon}\eta)$, where $\epsilon>0$ is a small constant and $\eta=N^{-2/3-\epsilon_0}, \epsilon_0>0$ will be chosen later; (ii). Replacing the sharp characteristic function getting from step (i) with a smooth cutoff function $q$ in terms of the Green function; (iii). Using the Green function comparison argument to compare the distribution of the singular vectors between the ensembles $X_G$ and $X_V.$ 

We will follow the proof strategy of \cite[Section 3]{KY} and slightly modify the detail. Specially, the choices of random interval in step (i) and the smooth function $q$ in step (ii) are different due to the fact that we have more than one bulk components. And the Green function comparison argument is also slightly different as we use the linearization matrix (\ref{main_representation}).

We  mainly focus on a single bulk component, firstly prove the singular vector distribution and then extend the results to singular values. The results containing several bulk components will follow after minor modification. We first prove the following result for the right singular vector.

\begin{lem} \label{lem_onedimensionalcase} For $Q_V=\Sigma^{1/2}X_VX_V^* \Sigma^{1/2}$ satisfying Assumption \ref{assu_main}, let  $\mathbb{E}^G, \mathbb{E}^V$ denote the expectations with respect to $X_G, X_V.$ Consider the $k$-th bulk component,  $k=1,2,\cdots,p,$ and $l$ defined in (\ref{defnalphaprime1}) or (\ref{defnalphaprime2}), under Assumption \ref{defnregularity} and \ref{assum_nozero},  for any choices of indices $\mu, \nu \in \mathcal{I}_2$, there exists a $\delta \in (0,1),$ when $l \leq N_k^{\delta},$ we have
\begin{equation*}
\lim_{N \rightarrow \infty} [\mathbb{E}^V-\mathbb{E}^G] \theta( N \zeta_{\alpha^{\prime}}(\mu) \zeta_{\alpha^{\prime}}(\nu))=0,
\end{equation*}
where $ \theta$ is a smooth function in $\mathbb{R}$ that satisfies
\begin{equation} \label{defn_theta1d}
\vert \theta^{(3)}(x) \vert \leq C_1 (1+ \vert x \vert)^{C_1}, \ x \in \mathbb{R},  \ \text{with some constant} \ C_1>0. 
\end{equation}
\end{lem}

Near the edges, by (\ref{thm_rigidity_equ}) and (\ref{delocalequation}), with $1- N^{-D_1}$ probability, we have
\begin{equation} \label{boundprob}
|\lambda_{\alpha^{\prime}}-\gamma_{\alpha^{\prime}}|\leq N^{-2/3+\epsilon_1},\ \max_{\mu,s_2} \vert \zeta_{\mu}(s_2) \vert^2 \leq N^{-1+\epsilon_1}.
\end{equation}
Hence, throughout the proofs of this section, we always use the scale parameter
\begin{equation} \label{edgeeta}
\eta =N ^{-2/3-\epsilon_0}, \ \epsilon_0>\epsilon_1 \ \text{is a small constant}.
\end{equation}

\begin{proof}[Proof of Lemma \ref{lem_onedimensionalcase}]
In a first step, we express the singular vector entries as an integral of Green functions over a random interval, which is recorded as the following lemma.   

\begin{lem} \label{step1} Under the assumptions of Lemma \ref{lem_onedimensionalcase},  there exists a small constant $0<\delta<1$, such that
\begin{equation} \label{lem_step1_eq}
\lim_{N \rightarrow \infty} \max_{l \leq N_k^{\delta} } \max_{\mu, \nu} \left \vert \mathbb{E}^V \theta(N \zeta_{\alpha^{\prime}}(\mu)\zeta_{\alpha^{\prime}}(\nu))-\mathbb{E}^V \theta[\frac{N}{\pi}\int_{I} \tilde{G}_{\mu \nu}(z) \mathcal{X}(E) dE] \right\vert =0,
\end{equation}
where $I$ is defined as 
\begin{equation}\label{defn_i_1}
I:=[a_{2k-1}-N^{-\frac{2}{3}+\epsilon}, \ a_{2k-1}+N^{-\frac{2}{3}+\epsilon}], 
\end{equation}
when (\ref{defnalphaprime1}) holds, and when  (\ref{defnalphaprime2}) holds, it is denoted as
\begin{equation}\label{defn_i_2}
I:=[a_{2k}-N^{-\frac{2}{3}+\epsilon}, \ a_{2k}+N^{-\frac{2}{3}+\epsilon}],
\end{equation}
with  $\epsilon $ satisfies that, for $C_1$ defined in (\ref{defn_theta1d})
\begin{equation} \label{defn_epsion}
2(C_1+1)(\delta+\epsilon_1)<\epsilon<c \epsilon_0, \ c>0 \text{ is a constant much smaller than} \ 1.
\end{equation} 
And $\mathcal{X}(E)$ is defined as
\begin{equation}\label{defn_x(e)}
\mathcal{X}(E):=\mathbf{1}( \lambda_{\alpha^{\prime}+1} < E^{-} \leq \lambda_{\alpha^{\prime}} ),
\end{equation}
where $E^{\pm}:=E \pm N^{\epsilon}\eta$.  The conclusion holds true if we replace $X_V$ with $X_G.$ 
\end{lem}

\begin{proof}
We first observe that 
\begin{equation*}
\zeta_{\alpha^{\prime}}(\mu)\zeta_{\alpha^{\prime}}(\nu)=\frac{\eta}{\pi} \int_{\mathbb{R}} \frac{ \zeta_{\alpha^{\prime}}(\mu)\zeta_{\alpha^{\prime}}(\nu)}{(E-\lambda_{\alpha^{\prime}})^2+\eta^2} dE.
\end{equation*}
Choose $a, \ b$ such that  
\begin{equation} \label{defn_ab}
a:=\min \{ \lambda_{\alpha^{\prime}}-N^{\epsilon}\eta, \ \lambda_{\alpha^{\prime}+1 }+N^{\epsilon}\eta \}, \  b:= \lambda_{\alpha^{\prime}}+N^{\epsilon}\eta. 
\end{equation} 
We also observe the elementary inequality (see the equation above (6.10) of \cite{EYY}), for some constant $C>0,$
\begin{equation} \label{anestimationequ}
\int_{x}^{\infty} \frac{\eta}{\pi (y^2+\eta^2)}dy \leq \frac{C \eta}{x+\eta}, \ x>0.
\end{equation}
\noindent By (\ref{boundprob}), (\ref{defn_ab}) and (\ref{anestimationequ}), with $1-N^{-D_1}$ probability, we have
\begin{equation} \label{greenfunctioneq1}
\zeta_{\alpha^{\prime}}(\mu)\zeta_{\alpha^{\prime}}(\nu)=\frac{\eta}{\pi} \int_a^b \frac{\zeta_{\alpha^{\prime}}(\mu)\zeta_{\alpha^{\prime}}(\nu)}{(E-\lambda_\alpha^{\prime})^2+\eta^2} dE+ O(N^{-1-\epsilon+\epsilon_1}).
\end{equation}
By  (\ref{defn_theta1d}),  (\ref{boundprob}), (\ref{defn_epsion}), (\ref{greenfunctioneq1}) and mean value theorem  , we have 
\begin{equation} \label{greenfunctioneq2}
\mathbb{E}^V \theta(N \zeta_{\alpha^{\prime}}(\mu)\zeta_{\alpha^{\prime}}(\nu) )= \mathbb{E}^V \theta(\frac{ N \eta}{\pi} \int_a^b \frac{\zeta_{\alpha^{\prime}}(\mu)\zeta_{\alpha^{\prime}}(\nu)}{(E-\lambda_{\alpha^{\prime}})^2+\eta^2} dE )+o(1).
\end{equation}
Denote $\lambda_{t}^{\pm}:=\lambda_t \pm N^{\epsilon}\eta, \ t=\alpha^{\prime}, \ \alpha^{\prime}+1$, and by (\ref{defn_ab}), we have 
\begin{equation*}
\int_a^b dE= \int_{ \lambda_{\alpha^{\prime}+1}^{+}}^{\lambda_{\alpha^{\prime}}^{+}} dE + \mathbf{1}(\lambda_{\alpha^{\prime}+1}^{+} > \lambda_{\alpha^{\prime}}^{-})\int_{\lambda_{\alpha^{\prime}}^{-}}^{\lambda_{\alpha^{\prime}+1}^{+}}dE.
\end{equation*}
\noindent By  (\ref{defn_theta1d}), (\ref{boundprob}), (\ref{greenfunctioneq2}) and mean value theorem , we have
\begin{equation}\label{greenfunction3}
\mathbb{E}^V \theta(N \zeta_{\alpha^{\prime}}(\mu)\zeta_{\alpha^{\prime}}(\nu))= \mathbb{E}^V \theta(\frac{ N \eta}{\pi} \int_{\lambda^{+}_{\alpha^{\prime}+1}}^{\lambda_{\alpha^{\prime}}^{+}} \frac{\zeta_{\alpha^{\prime}}(\mu)\zeta_{\alpha^{\prime}}(\nu)}{(E-\lambda_{\alpha^{\prime}})^2+\eta^2} dE)+o(1),
\end{equation}
 where we use (\ref{thm_rigidity_equ}) and (\ref{defn_epsion}). Next we can without loss of generality, consider the case when (\ref{defnalphaprime1}) holds true.
By (\ref{boundprob}) and (\ref{defn_epsion}), we observe that with $1-N^{-D_1}$ probability, we have $\lambda_{\alpha^{\prime}}^{+} \leq a_{2k-1}+N^{-2/3+\epsilon}$ and  $\lambda_{\alpha^{\prime}+1}
^{+} \geq a_{2k-1}-N^{-2/3+\epsilon}.$  By (\ref{thm_rigidity_equ})  and the choice of $I$ in (\ref{defn_i_1}), we have 
\begin{equation} \label{greenfunction4}
\mathbb{E}^V \theta(N \zeta_{\alpha^{\prime}}(\mu)\zeta_{\alpha^{\prime}}(\nu))= \mathbb{E}^V \theta(\frac{ N \eta}{\pi} \int_I  \frac{\zeta_{\alpha^{\prime}}(\mu)\zeta_{\alpha^{\prime}}(\nu)}{(E-\lambda_{\alpha^{\prime}})^2+\eta^2} \mathcal{X}(E) dE )+o(1).
\end{equation}
Recall (\ref{keykeykeyindentity}), we can split the summation as
\begin{equation}\label{main_representation1}
\frac{1}{\eta } \tilde{G}_{\mu \nu}(z)=\sum_{\beta \neq \alpha^{\prime}}  \frac{\zeta_{\beta}(\mu)\zeta_\beta(\nu)}{(E-\lambda_{\beta})^2+\eta^2}+ \frac{\zeta_{\alpha^{\prime}}(\mu)\zeta_{\alpha^{\prime}}(\nu)}{(E-\lambda_{\alpha^{\prime}})^2+\eta^2}. 
\end{equation}
Denote $\mathcal{A}:=\{\beta \neq \alpha^{\prime}: \lambda_{\beta} \ \text{is not in the $k$-th bulk component} \}.$
By (\ref{boundprob}), with $1-N^{-D_1}$ probability, we have
\begin{equation}\label{main_representation2}
\left \vert \sum_{\beta \neq \alpha^{\prime} } \frac{N \eta}{\pi} \int_{I} \frac{\zeta_{\beta}(\mu)\zeta_{\beta}(\nu)}{(E-\lambda_{\beta})^2+\eta^2} dE \right \vert \leq  \frac{N^{\epsilon_1}}{\pi} \left( \sum_{\beta \in \mathcal{A} } \int_{I} \frac{\eta}{\eta^2+(E-\lambda_{\beta})^2} dE+\sum_{\beta \in \mathcal{A}^c } \int_{I} \frac{\eta}{\eta^2+(E-\lambda_{\beta})^2} dE \right).
\end{equation}
By Assumption \ref{defnregularity}, with $1-N^{D_1}$ probability,  we have
\begin{equation} \label{outbulk}
\frac{N^{\epsilon_1}}{\pi} \sum_{\beta \in \mathcal{A}} \int_{I} \frac{\eta}{\eta^2+(E-\lambda_{\beta})^2} dE \leq N^{\epsilon_1} \sum_{\beta \in \mathcal{A} } N^{-4/3-\epsilon_0+\epsilon}.
\end{equation}
Denote 
\begin{equation}\label{defn_l1}
l(\beta):=\beta-\sum_{t<k} N_t.
\end{equation}
By (\ref{boundprob}), with $1- N^{-D_1}$ probability, for some small constant $0<\delta<1,$ we have
\begin{equation}\label{main_representation22}
\frac{N^{\epsilon_1}}{\pi} \sum_{\beta \in \mathcal{A}^c }  \int_{I} \frac{\eta}{(E-\lambda_{\beta})^2+\eta^2} dE  \leq N^{\epsilon_1+\delta}+ \frac{1}{\pi} \sum_{\beta \in \mathcal{A}^c; \ l(\beta) \geq N_k^{\delta}} \int_{I} \frac{N^{\epsilon_1}\eta}{\eta^2+(E-\lambda_{\beta})^2} dE.
\end{equation}
By Assumption \ref{defnregularity}, (\ref{squarerootequation}) and (\ref{thm_rigidity_equ}), it is easy to check that (see (3.12) of \cite{KY})
\begin{equation} \label{Ebound}
(E-\lambda_{\beta})^2 \geq c(\frac{l(\beta)}{N})^{4/3}, \ c>0 \ \text{is some constant}.
\end{equation}
\noindent By  (\ref{Ebound}),  with  $1-N^{-D_1}$ probability, we have
\begin{align*}
\frac{1}{\pi}\sum_{\beta \in \mathcal{A}^c; \ l(\beta) \geq N_k^{\delta}} \int_{I} \frac{ N^{\epsilon_1}  \eta}{\eta^2+(E-\lambda_{\beta})^2} dE  \leq N^{\epsilon_1-\epsilon_0+\epsilon} \int_{N^{\delta}-1}^N \frac{1}{x^{4/3}}dx \leq N^{-\delta/3+\epsilon_1-\epsilon_0+\epsilon}.
\end{align*}
\noindent Recall (\ref{defn_epsion}), we can restrict $\epsilon_1-\epsilon_0+\epsilon<0$, with   $1-N^{-D_1}$ probability, this yields 
\begin{equation} \label{greenfunctionbound1}
\sum_{\beta \in \mathcal{A}^c; \ l(\beta) \geq N_k^{\delta}} \int_{I} \frac{ N^{\epsilon_1}  \eta}{\eta^2+(E-\lambda_{\beta})^2} dE \leq N^{-\delta/3}. 
\end{equation}
By (\ref{main_representation2}), (\ref{outbulk}), (\ref{main_representation22}) and (\ref{greenfunctionbound1}), with $1-N^{-D_1}$ probability,  we have
\begin{equation} \label{greenfunctionbound2}
\left \vert \sum_{\beta \neq \alpha^{\prime} } \frac{N \eta}{\pi} \int_{I} \frac{\zeta_{\beta}(\mu)\zeta_{\beta}(\nu)}{(E-\lambda_{\beta})^2+\eta^2} dE \right \vert \leq N^{\delta+2\epsilon_1}.
\end{equation}
\noindent By (\ref{defn_theta1d}), (\ref{boundprob}), (\ref{main_representation1}), (\ref{greenfunctionbound2}) and mean value theorem, we have 
\begin{align} \label{MVT}
\left| \mathbb{E}^V \theta(\frac{N \eta}{\pi}\int_{I} \frac{\zeta_{\alpha^{\prime}}(\mu)\zeta_{\alpha^{\prime}}(\nu)}{(E-\lambda_{\alpha^{\prime}})^2+\eta^2} \mathcal{X}(E) dE)- \mathbb{E}^V \theta(\frac{N }{\pi } \int_{I} \tilde{G}_{\mu \nu}(E+i\eta)\mathcal{X}(E)dE ) \right | \nonumber \\ \leq N^{C_1(\delta+2\epsilon_1)} \mathbb{E}^V \sum_{\beta \neq \alpha^{\prime}} \frac{N \eta}{\pi} \int_{I} \frac{\vert \zeta_{\beta}(\mu)\zeta_{\beta}(\nu) \vert}{(E-\lambda_{\beta})^2+\eta^2} \mathcal{X}(E) dE,
\end{align}
where $C_1$ is defined in (\ref{defn_theta1d}).  To finish the proof, it suffices to estimate the right-hand side of (\ref{MVT}). Similar to (\ref{outbulk}), we have
\begin{equation}\label{outbulk1}
 \sum_{\beta \in \mathcal{A} } \int_{I} \frac{\eta}{\eta^2+(E-\lambda_{\beta})^2} dE \leq N^{-1/3-\epsilon_0+\epsilon}. 
\end{equation}

\noindent Choose a small constant $0<\delta_1<1$, repeat the estimation of  (\ref{greenfunctionbound1}), we have
\begin{equation}\label{truncationeq1}
\sum_{\beta \in \mathcal{A}^c ; \ l(\beta) \geq N_k^{\delta_1}}  \int_{I} \frac{\eta}{\eta^2+(E-\lambda_{\beta})^2} dE \leq N^{-\delta_1/3+\epsilon-\epsilon_0}.
\end{equation}
Recall (\ref{defnalphaprime1}) and restrict $\epsilon>2((C_1+1)\epsilon_1+\delta_1+C_1\delta)$, by (\ref{boundprob}) and (\ref{anestimationequ}), we have
\begin{align} \label{truncarioneq2}
\sum_{\beta \in \mathcal{A}^c ;\  l \leq
 l(\beta) \leq N_k^{\delta_1}} \frac{N \eta}{\pi} \mathbb{E}^V \int_{I}  \frac{\vert \zeta_{\beta}(\mu)\zeta_{\beta}(\nu) \vert}{(E-\lambda_{\beta})^2+\eta^2} \mathcal{X}(E) dE  & \leq \mathbb{E}^V \int_{\lambda_{\alpha^{\prime}+1}+N^{\epsilon}\eta}^{\infty}\frac{ N^{\delta_1+\epsilon_1} \eta}{(E-\lambda_{\alpha^{\prime}+1})^2+\eta^2}dE \nonumber \\ 
&   \leq N^{-\epsilon+\epsilon_1+\delta_1},
\end{align}
where we use the fact that $\beta \in \mathcal{A}^c $ and $\  l < l(\beta) \leq N_k^{\delta_1}$  implies $\lambda_{\beta} \leq \lambda_{\alpha^{\prime}+1}$. It remains to estimate the summation of the terms when $\beta \in \mathcal{A}^c$ and $l(\beta) < l.$  For a given constant $\epsilon^{\prime}$ satisfies
\begin{equation}\label{key_ris}
 \frac{1}{2}(\epsilon_0+3\epsilon+2(C_1+1)\epsilon_1+(C_1+1)\delta)<\epsilon^{\prime}<\epsilon_0,
 \end{equation}
 we partition $I=I_1 \cup I_2 $ with $I_1 \cap I_2=\emptyset$ by denoting
\begin{equation} \label{defn_i1}
I_1: = \{ E \in I: \exists \beta,\beta \in \mathcal{A}^c, \  l(\beta) < l, \ \vert E-\lambda_{\beta} \vert \leq N^{\epsilon^{\prime}}\eta  \}.
\end{equation}
By (\ref{boundprob}) and (\ref{defn_i1}), we have 
\begin{equation}\label{partitioni2}
\sum_{\beta \in \mathcal{A}^c;\  l(\beta) < l} \frac{N \eta}{\pi} \mathbb{E}^V \int_{I_2}  \frac{\vert \zeta_{\beta}(\mu)\zeta_{\beta}(\nu) \vert}{(E-\lambda_{\beta})^2+\eta^2} \mathcal{X}(E) dE \leq N^{-2 \epsilon^{\prime}+\epsilon_0+\epsilon+\epsilon_1+\delta}.
\end{equation}
It is easy to check that on $I_1$ when $\lambda_{\alpha^{\prime}+1}\leq \lambda_{\alpha^{\prime}}<\lambda_{\beta}$, we have (see (3.15) of \cite{KY}) 
\begin{equation} \label{eigenvalueestimate}
\frac{1}{(E-\lambda_{\beta})^2+\eta^2} \mathbf{1}(E^{-} \leq \lambda_{\alpha^{\prime}} ) \leq \frac{N^{2\epsilon}}{(\lambda_{\alpha^{\prime}+1}-\lambda_{\alpha^{\prime}})^2+\eta^2}. 
\end{equation}
By (\ref{boundprob}) and (\ref{eigenvalueestimate}), we have
\begin{align} \label{truncationseq3}
\sum_{\beta \in \mathcal{A}^c;   l(\beta) \leq l} \frac{N \eta}{\pi} \mathbb{E}^V & \int_{I_1}  \frac{\vert \zeta_{\beta}(\mu)\zeta_{\beta}(\nu) \vert}{(E-\lambda_{\beta})^2+\eta^2} \mathcal{X}(E) dE  \leq \mathbb{E}^V \int_{I_1} \frac{N^{\delta+\epsilon_1+2\epsilon-2/3}\eta}{(\lambda_{\alpha^{\prime}+1}-\lambda_{\alpha^{\prime}})^2+\eta^2} dE \nonumber \\
& \leq N^{\delta+\epsilon_1+3\epsilon-D_1+2/3+\epsilon_0}+ N^{-2\epsilon^{\prime}+\epsilon_0+\epsilon_1+\delta+3\epsilon}.
\end{align}
By (\ref{outbulk1}), (\ref{truncationeq1}), (\ref{truncarioneq2}), (\ref{key_ris}) and (\ref{truncationseq3}), we conclude the proof of (\ref{MVT}). It is clear that our proof still applies when we replace $X_V$ with $X_G.$ 
\end{proof}

In a second step, we will write the sharp indicator function of (\ref{defn_x(e)}) as some smooth function $q$ of $\tilde{G}_{\mu \nu}$. To be consistent with the proof of Lemma \ref{step1}, we consider the bulk edge $a_{2k-1}.$ Denote 
\begin{equation}\label{thetaeta}
\vartheta_{\eta}(x):=\frac{\eta}{\pi(x^2+\eta^2)}=\frac{1}{\pi}\operatorname{Im} \frac{1}{x-i\eta}.
\end{equation}
We define a smooth cutoff function $q \equiv q_{\alpha^{\prime}}: \mathbb{R} \rightarrow \mathbb{R}_{+}$ satisfying
\begin{equation} \label{defn_q}
q(x)=1, \ \text{if}  \ \  |x-l| \leq \frac{1}{3}; \ \  \ q(x)=0, \ \text{if} \ \   |x-l\vert\geq \frac{2}{3},
\end{equation}
where $l$ is defined in (\ref{defnalphaprime1}). We also denote $Q_1=Y^*Y.$
\begin{lem} \label{truncation}
For $\epsilon$ defined in (\ref{defn_epsion}) , denote
\begin{equation}  \label{keywindowlength}
\mathcal{X}_{E}(x):=\mathbf{1}(E^{-} \leq x \leq E_U),
\end{equation}
where $E_U:=a_{2k-1}+2N^{-2/3+\epsilon}$.  Denote $\tilde{\eta}:=N^{-2/3-9\epsilon_0},$ where $\epsilon_0$ is defined in (\ref{edgeeta}),  we have
\begin{equation} \label{step2equ}
\lim_{N \rightarrow \infty} \max_{l \leq N_k^{\delta}} \max_{\mu,\nu} \left |  \mathbb{E}^V \theta(N \zeta_{\alpha^{\prime}}(\mu)\zeta_{\alpha^{\prime}}(\nu))-\mathbb{E}^V \theta(\frac{N}{\pi } \int_I \tilde{G}_{\mu \nu}(z)q[\operatorname{Tr}(\mathcal{X}_E*\vartheta_{\tilde{\eta}})(Q_1)]dE) \right |=0,
\end{equation} 
where $I$ is defined in (\ref{defn_i_1}) and $*$ is the convolution operator.
\end{lem}
\begin{proof}
For any $E_1 <E_2$, denote the number of eigenvalues of $Q_1$ in $[E_1, E_2]$ by
\begin{equation} \label{defn_countingfunction}
\mathcal{N}(E_1, E_2):=\#\{j: E_1 \leq \lambda_j \leq E_2\}.
\end{equation}
Recall (\ref{defn_i_1}) and (\ref{defn_x(e)}), it is easy to check that with $1-N^{-D_1}$ probability, we have 
\begin{align} \label{part_estimate}
N \int_I \tilde{G}_{\mu \nu}(z)\mathcal{X}(E) dE=N \int_{I} \tilde{G}_{\mu \nu}(z)\mathbf{1}(\mathcal{N}(E^{-}, E_U)=l)d E 
=N \int_{I} \tilde{G}_{\mu \nu}(z)q[\operatorname{Tr}\mathcal{X}_{E}(Q_1)]dE, 
\end{align}
where for the second equality, we use (\ref{thm_rigidity_equ}) and Assumption \ref{defnregularity}. We use the following lemma to estimate (\ref{defn_countingfunction}) by its delta approximation smoothed on the scale $\tilde{\eta}.$  The proof is put in Appendix \ref{appendix_a}. 
\begin{lem} \label{lem_deltaestimatetileeta}For $t=N^{-2/3-3\epsilon_0},$ there exists some constant $C$, with $1-N^{-D_1}$ probability,  for any $E$ satisfying 
\begin{equation} \label{defn_edomain}
|E^{-}-a_{2k-1}| \leq \frac{3}{2}N^{-2/3+\epsilon},
\end{equation}
we have
\begin{equation}\label{lem_deltaestimate}
| \operatorname{Tr} \mathcal{X}_E(Q_1)-\operatorname{Tr}(\mathcal{X}_E*\vartheta_{\tilde{\eta}})(Q_1) | \leq C(N^{-2\epsilon_0}+\mathcal{N}(E^--t, \ E^-+t)).
\end{equation}
\end{lem}
\noindent By (A.7) of \cite{KY1}, for any $ z \in \mathbf{D}(\tau)$ defined in (\ref{DOMAIN1}), we have
\begin{equation} \label{squareroot}
 \operatorname{Im} m(z) \sim \begin{cases} 
      \eta/\sqrt{\kappa+\eta} & E \notin \operatorname{supp}(\rho), \\
    \sqrt{\kappa+\eta} & E \in \operatorname{supp}(\rho).
   \end{cases},
\end{equation} 
where $\kappa:=|E-a_{2k-1}|.$  When $\mu=\nu,$ with $1-N^{-D_1}$ probability, we have
\begin{equation*}
\sup_{E \in I} |\tilde{G}_{\mu \mu}(E+i \eta)|=\sup_{E \in I}|\operatorname{Im} G_{\mu \mu}(z)| \leq \sup_{E \in I}(\operatorname{Im}|G_{\mu \mu}(z)-m_{}(z)|+|\operatorname{Im}m_{}(z)|) \leq N^{-1/3+\epsilon_0+2\epsilon},
\end{equation*}
where we use (\ref{thm_anisotropic_eq12}) and (\ref{squareroot}). When $\mu \neq \nu,$ we use the following identity 
\begin{equation*}
 \tilde{G}_{\mu \nu}=\eta \sum_{k=M+1}^{M+N} G_{\mu k} \overline{G}_{\nu k}.
\end{equation*}
By (\ref{thm_anisotropic_eq12}) and (\ref{squareroot}), with $1-N^{-D_1}$ probability, we have $\sup_{E \in I} |\tilde{G}_{\mu \nu}(z)| \leq N^{-1/3+\epsilon_0+2\epsilon}. $ Therefore,  for $E \in I,$ with $1-N^{-D_1}$ probability,  we have
\begin{equation}\label{Ebounds}
\sup_{E \in I} |\tilde{G}_{\mu \nu}(E+i \eta)| \leq  N^{-1/3+3\epsilon_0/2}.
\end{equation}
\noindent Recall (\ref{defn_q}), by (\ref{part_estimate}),  (\ref{lem_deltaestimate}),  (\ref{Ebounds}) and the smoothness of $q$, with $1- N^{-D_1}$ probability, we have
\begin{align} \label{truncation_estimate_eq}
\left |  N \int_I \tilde{G}_{\mu \nu}(z)\mathcal{X}(E) dE- N \int_{I} \tilde{G}_{\mu \nu}(z)q[\operatorname{Tr}(\mathcal{X}_E*\vartheta_{\tilde{\eta}} (Q_1))] dE   \right |
& \leq   CN \sum_{ l(\beta) \leq N_k^{\delta}} \int_{I} \vert \tilde{G}_{\mu \nu}(z) \vert \mathbf{1}(\vert E^--\lambda_{\beta}\vert \leq t ) dE+N^{-\epsilon_0/4}  \nonumber \\
& \leq CN^{1+\delta} |t| \sup_{z \in I} \vert \tilde{G}_{\mu \nu}(z) \vert+ N^{-\epsilon_0/4}.
\end{align}
By (\ref{Ebounds})  and (\ref{truncation_estimate_eq}),  we have 
\begin{equation*}
\left |  N \int_I \tilde{G}_{\mu \nu}(z)\mathcal{X}(E) dE- N \int_{I} \tilde{G}_{\mu \nu}(z)q[\operatorname{Tr}(\mathcal{X}_E*\vartheta_{\tilde{\eta}}(Q_1))] dE   \right | \leq CN^{-\epsilon_0/2+\delta}+N^{-\epsilon_0/4}.
\end{equation*}
 Using a similar discussion to (\ref{main_representation2}), by (\ref{defn_theta1d}) and (\ref{defn_epsion}),  we finish the proof.
\end{proof}

In the final step, we use the Green function comparison argument to prove the following lemma, whose proof will be put in Section \ref{greenfunction_comp_edge}.
\begin{lem}\label{greenfunction_comp}
Under the assumptions of Lemma \ref{truncation}, we have
\begin{equation*}
\lim_{N \rightarrow \infty} \max_{\mu, \nu} (\mathbb{E}^V-\mathbb{E}^G) \theta \left(\frac{N}{\pi } \int_I \tilde{G}_{\mu \nu}(z)q[\operatorname{Tr}(\mathcal{X}_E*\vartheta_{\tilde{\eta}})(Q_1)]dE\right)=0.  
\end{equation*}
\end{lem}
Once Lemma \ref{greenfunction_comp} is proved, the proof of Lemma \ref{lem_onedimensionalcase} follows from Lemma \ref{truncation}.
\end{proof}

\subsection{Green function comparsion argument} \label{greenfunction_comp_edge}
In this section, we will prove Lemma \ref{greenfunction_comp} using the Green function comparison argument.  In the end of this section, we will discuss how we can extend Lemma \ref{lem_onedimensionalcase} to Theorem \ref{thm_edge} and Theorem \ref{thm_edge_sev}. By the orthonormal properties of $\xi, \zeta$ and (\ref{main_representation}), we have
\begin{equation} \label{eigenvectorrepresentation}
\tilde{G}_{ij}=\eta\sum_{k=1}^M G_{ik} \overline{G}_{jk}, \ \tilde{G}_{\mu \nu}=\eta \sum_{k=M+1}^{M+N} G_{\mu k} \overline{G}_{\nu k}.
\end{equation} 
\noindent By (\ref{thm_anisotropic_eq12}), with $1-N^{-D_1}$ probability, we have
\begin{equation} \label{edgebound}
\vert G_{\mu \mu} \vert =O(1), \ \vert G_{\mu \nu} \vert \leq N^{-1/3+2\epsilon_0}, \ \mu \neq \nu.
\end{equation}
We firstly drop the all diagonal terms in (\ref{eigenvectorrepresentation}).
\begin{lem} \label{dropdiagonal}  Recall $E_U=a_{2k-1}+2N^{-2/3+\epsilon}$ and $\tilde{\eta}=N^{-2/3-9\epsilon_0}$, we have
\begin{equation} \label{greenfunctionreduced1}
\mathbb{E}^V \theta \left[\frac{N}{\pi} \int_I \tilde{G}_{\mu\nu}(z)q[\operatorname{Tr}(\mathcal{X}_E* \vartheta_{\tilde{\eta}})(Q_1)]dE\right]-\mathbb{E}^V\theta \left[\int_I x(E) q(y(E))dE\right]=o(1), 
\end{equation}
where we denote $X_{\mu\nu,k}:=G_{\mu k}\overline{G}_{\nu k}$ and 
\begin{equation} \label{reduction}
x(E):=\frac{N \eta}{\pi} \sum_{k=M+1, \ \text{and} \  \neq \mu, \nu}^{M+N} X_{\mu\nu,k}(E+i\eta), \ y(E):=\frac{\tilde{\eta}}{\pi} \int_{E^{-}}^{E_U} \sum_k \sum_{\beta \neq k}X_{\beta \beta,k}(E+i\tilde{\eta})dE.
\end{equation}
The conclusion holds true if we replace $X_V$ with $X_G.$
\end{lem}
\begin{proof}
We first observe that by (\ref{edgebound}),  with $1-N^{-D_1}$ probability,  we have
\begin{equation}\label{taylor_bdd1}
\vert x(E) \vert  \leq  N^{2/3+3\epsilon_0},
\end{equation}
which implies that
\begin{equation} \label{taylor_bdd2}
\int_{I} \vert x (E)\vert dE \leq N^{4\epsilon_0}.
\end{equation}
By (\ref{eigenvectorrepresentation}) and (\ref{edgebound}), with $1-N^{-D_1}$ probability,  we have 
\begin{equation} \label{keyformula0}
\left |  \frac{N}{\pi} \tilde{G}_{\mu \nu}(E+i\eta)-x(E)   \right |=\frac{N \eta}{\pi}\vert G_{\mu \mu} \overline{G}_{\nu  \mu}+G_{\mu \nu} \overline{G}_{\nu \nu}\vert \leq N \eta (\mathbf{1}(\mu=\nu)+N^{-1/3+2\epsilon_0}\mathbf{1}(\mu \neq \nu)).
\end{equation}
By the equations (5.11) and (6.42) of \cite{DY}, we have
\begin{equation}\label{keyformula}
\operatorname{Tr}(\mathcal{X}_E*\vartheta_{\tilde{\eta}}(Q_1))=\frac{N}{\pi} \int_{E^-}^{E_U} \operatorname{Im} m_2(w+i\tilde{\eta})dw, \ \ \sum_{\mu \nu} \vert G_{\mu \nu} (w+i\tilde{\eta})\vert^2=\frac{N \operatorname{Im}m_2(w+i\tilde{\eta})}{\tilde{\eta}}.
\end{equation}
Therefore, we have
\begin{equation} \label{differencetrqy}
\operatorname{Tr}(\mathcal{X}_E*\vartheta_{\tilde{\eta}}(Q_1))-y(E)=\frac{\tilde{\eta}}{\pi}\int_{E^{-}}^{E_U} \sum_{\beta=M+1}^{M+N} |G_{\beta \beta}|^2 dw.
\end{equation}
By (\ref{differencetrqy}), mean value theorem and  the fact $q$ is smooth enough, we have 
\begin{equation} \label{keyformula2}
\left |  q[\operatorname{Tr}(\mathcal{X}_E* \vartheta_{\tilde{\eta}})(Q_1)]- q[y(E)]  \right | \leq N^{-1/3-7\epsilon_0}.
\end{equation}
Therefore, by mean value theorem, (\ref{defn_theta1d}), (\ref{defn_epsion}), (\ref{taylor_bdd1}), (\ref{taylor_bdd2}), (\ref{keyformula0}) and (\ref{keyformula2}),  we can conclude our proof. 
\end{proof}

To prove Lemma \ref{greenfunction_comp}, by (\ref{greenfunctionreduced1}), it suffices to prove 
\begin{equation} \label{lemma34reduced}
[\mathbb{E}^V-\mathbb{E}^G] \theta[\int_I x(E)q(y(E))dE]=o(1).
\end{equation}
For the rest, we will use the Green function comparison argument to prove (\ref{lemma34reduced}), where we follow the basic approach of \cite[Section 6]{DY} and \cite[Section 3.1]{KY}. Define a bijective ordering map $\Phi$ on the index set, where
\begin{equation*}
\Phi: \{(i,\mu_1):1 \leq i \leq M, \ M+1 \leq \mu_1 \leq M+N \} \rightarrow \{1,\ldots,\gamma_{\max}=MN\}.
\end{equation*} 
Recall that we relabel $X^V =((X_V)_{i \mu_1}, i \in \mathcal{I}_1, \mu_1 \in \mathcal{I}_2)$, similarly for $X^G.$ For any $1\le \gamma \le \gamma_{\max}$, we define the matrix $X_{\gamma}= \left(x^{\gamma}_{i\mu_1}\right)$ such that $x_{i \mu_1}^{\gamma} =X^G_{i\mu_1} $ if $\Phi(i,\mu_1)> \gamma$, and $x_{i\mu_1}^{\gamma} =X^V_{i\mu_1}$ otherwise. Note that $X_0=X^G$ and $X_{\gamma_{\max}}=X^V$. With the above definitions, we have 
\begin{equation*}
 [\mathbb{E}^G-\mathbb{E}^V] \theta[\int_I x(E)q(y(E))dE]=\sum_{\gamma=1}^{\gamma_{\max}}[\mathbb{E}^{\gamma-1}-\mathbb{E}^{\gamma}]\theta[\int_I x(E)q(y(E))dE].
\end{equation*}
For simplicity,  we rewrite the above equation as
\begin{equation} \label{keyreduceprof}
\mathbb{E}[\theta(\int_I x^G q(y^G)dE)-\theta(\int_I x^V q(y^V)dE)]=\sum_{\gamma=1}^{\gamma_{\max}}\mathbb{E}[\theta(\int_I x_{\gamma-1} q(y_{\gamma-1})dE)-\theta(\int_I x_{\gamma} q(y_{\gamma})dE)].
\end{equation}
The key step of the Green function comparison argument is to use Lindeberg replacement strategy and find an intermediate variable $\Omega \equiv \Omega(\gamma)$  satisfying $\Omega(\gamma)=\Omega(\gamma-1)$, such that we can write
\begin{equation*}
\theta(\int_I x_{\gamma-1} q(y_{\gamma-1})dE)-\theta(\int_I x_{\gamma} q(y_{\gamma})dE)=\left(\theta(\int_I x_{\gamma-1} q(y_{\gamma-1})dE)-\Omega(\gamma-1)\right)+\left(\theta(\int_I x_{\gamma} q(y_{\gamma})dE)-\Omega(\gamma)\right),
\end{equation*}
and  $\theta(\int_I x_{t} q(y_{t})dE)-\Omega(t), \ t=\gamma-1, \ \gamma$ are small enough. We focus on the indices $s,t \in \mathcal{I}$, the special case $\mu, \ \nu \in \mathcal{I}_2$ follow. 
Denote $Y_{\gamma}:=\Sigma^{1/2}X_{\gamma}$ and 
 \begin{equation}\label{Hgamma}
   H^{\gamma} := \left( {\begin{array}{*{20}c}
   { 0 } & z^{1/2}Y_{\gamma} \\
   z^{1/2}Y^{*}_{\gamma} & {0}  \\
   \end{array}} \right), \ \ G^\gamma:= \left( {\begin{array}{*{20}c}
   { -zI} & z^{1/2}Y_{\gamma}  \\
   {z^{1/2}Y^*_{\gamma}} & { - zI}  \\
\end{array}} \right)^{-1}.
 \end{equation}
As $\Sigma$ is diagonal,  for each fixed $\gamma$,  $\ H^{\gamma}$ and $H^{\gamma-1}$ differ only at $(i,\mu_1)$ and $(\mu_1,i)$ elements, where $\Phi(i,\mu_1) = \gamma$. Then we define the $(N+M)\times (N+M)$ matrices $V$ and $W$ by
$$V_{ab}=z^{1/2}\left(\mathbf{1}_{\{(a,b)=(i,\mu_1)\}} + \mathbf{1}_{\{(a,b)=(\mu_1,i)\}}\right) \sqrt{\sigma_i} X^G_{i\mu_1}, \ \ W_{ab}=z^{1/2}\left(\mathbf{1}_{\{(a,b)=(i,\mu_1)\}}+ \mathbf{1}_{\{(a,b)=(\mu_1,i)\}}\right) \sqrt{\sigma_i} X^V_{i\mu_1},$$
so that $H^{\gamma}$ and $H^{\gamma-1}$ can be written as
$$H^{\gamma-1}= O+ V, \ \ H^{\gamma} = O+W,$$
for some $(N+M)\times (N+M)$ matrix $O$ satisfying $O_{i\mu_1}=O_{\mu_1 i}=0$ and $O$ is independent of $V$ and $W$. We will take $O$ as our intermediate variable. Denote 
\begin{equation}\label{srt}
S:=(H^{\gamma-1}-z)^{-1}, \ R:=(O-z)^{-1}, \ T:=(H^{\gamma}-z)^{-1}.
\end{equation}
With the above definitions, we can write
\begin{equation}\label{finalgreenfunction}
\mathbb{E}[\theta(\int_I x^G q(y^G)dE)-\theta(\int_I x^V q(y^V)dE)]=\sum_{\gamma=1}^{\gamma_{\max}}\mathbb{E}[\theta(\int_I x^{S} q(y^{S})dE)-\theta(\int_I x^{T} q(y^{T})dE)].
\end{equation}
The comparison argument is based on the following resolvent expansion
\begin{equation} \label{resolvent}
S=R-RVR+(RV)^2R-(RV)^3R+(RV)^4S.
\end{equation}

\noindent For any integer $ m>0,$ by (6.11) of \cite{DY}, we have 
\begin{equation} \label{combinations}
([RV]^m R)_{ab}= \sum_{(a_i,b_i) \in \{(i,\mu_1), (\mu_1,i)\}:1 \leq i \leq m} (z)^{m/2}(\sigma_i)^{m/2} (X^G_{i\mu_1})^m R_{a a_1} R_{b_1 a_2} \cdots R_{b_m b} ,
\end{equation}
\begin{equation} \label{combinations1}
([RV]^m S)_{ab}= \sum_{(a_i,b_i) \in \{(i,\mu_1), (\mu_1,i)\}:1 \leq i \leq m} (z)^{m/2}(\sigma_i)^{m/2} (X^G_{i\mu_1})^m R_{a a_1} R_{b_1 a_2} \cdots S_{b_m b}.
\end{equation}
Denote
\begin{equation}\label{defn_deltax}
 \Delta X_{\mu \nu,k}:= S_{\mu k} \overline{S}_{\nu k}-R_{\mu k} \overline{R}_{\nu k}. 
\end{equation}
In \cite{KY}, the discussion relies on a crucial parameter (see (3.32) of \cite{KY}), which counts the maximum number of diagonal resolvent elements in $\Delta X_{\mu \nu, k}$. We will follow this strategy but using a different counting parameter and furthermore use (\ref{combinations}) and (\ref{combinations1}) as our key ingredients. Our discussion is slightly easier due to the loss of a free index (i.e. $i \neq \mu_1$). 

 Inserting (\ref{resolvent}) into (\ref{defn_deltax}), by (\ref{combinations}) and (\ref{combinations1}), we find that there exists a random variable $A_1$, which depends on the randomness only through $O$ and the first two moments of $X^G_{i \mu_1}$.  Taking the partial expectation with respect to the $(i, \mu_1)$-th entry of $X^G$(recall they are i.i.d), by (\ref{INTRODUCTIONEQ}), we have the following result.
\begin{lem} \label{lem_random_bound}
Recall (\ref{defn_psi}) and denote $\mathbb{E}_{\gamma}$ as the partial expectation with respect to $X^G_{i \mu_1}$, there exists some constant $C>0,$ with $1-N^{-D_1}$ probability, we have 
\begin{equation} \label{random1}
\left| \mathbb{E}_{\gamma}\Delta X_{\mu \nu,k}-A_1 \right | \leq N^{-3/2+C\epsilon_0} \Psi(z)^{3-s} , \ M+1 \leq k \neq  \mu, \nu \leq M+N,
\end{equation}
where $s$ counts the maximum number of  resolvent elements in $\Delta X_{\mu \nu, k}$ involving the index $\mu_1$ and defined as
\begin{equation} \label{defns}
s:=\mathbf{1}( ( \{ \mu, \nu\} \cap \{ \mu_1 \} \neq \emptyset) \cup (\{ k=\mu_1\})).
\end{equation}
\end{lem}

\begin{proof}
Inserting (\ref{resolvent}) into (\ref{defn_deltax}), the terms in the expansion containing $X^G_{i \mu_1}, \ (X^G_{i \mu_1})^2$ will be included in $A_1$, 
we only consider the terms containing $(X^G_{i \mu_1})^m, m \geq 3.$ We consider $m=3$ and discuss the following terms,
\begin{equation*}
 R_{\mu k}\overline{[(RV)^3 R]}_{\nu k}, \ [RVR]_{\mu k}\overline{[(RV)^2R]}_{\nu k}.
\end{equation*}
By (\ref{combinations}), we have 
\begin{equation} \label{combi1}
R_{\mu k}\overline{[(RV)^3R]}_{\nu k}= R_{\mu k} (\sum  (\sigma_i)^{3/2} (X^G_{i \mu_1})^3\overline{(z)^{3/2}R_{\nu a_1}R_{b_1 a_2} R_{b_2 a_3} R_{b_3 k}}.
\end{equation}
In the worst scenario, $R_{b_1 a_2}$ and $ R_{b_2 a_3}$ are assumed to be the diagonal entries of $R$.  Similarly, we have 
\begin{equation} \label{sum1}
[RVR]_{\mu k}\overline{[(RV)^2R]}_{\nu k}= (\sum z^{1/2}\sigma_i^{1/2} X_{i \mu_1}^G R_{\mu a_1} R_{b_1 k} )(\sum \sigma_i(X_{i \mu_1}^G)^2 \overline{zR_{\nu a_1} R_{b_1 a_2} R_{b_2 k}} ),
\end{equation}
and the worst scenario is the case when $R_{b_1 a_2}$ is a diagonal term.  As $ \mu, \nu \neq i$ is always true and there are only finite terms of summation, by (\ref{INTRODUCTIONEQ}) and (\ref{edgebound}), for some constant $C$, we have 
\begin{equation*}
\mathbb{E}_{\gamma} |R_{\mu k}\overline{[(RV)^3R]}_{\nu k}| \leq  N^{-3/2+C\epsilon_0} \Psi(z)^{3-s}.  
\end{equation*}
Similarly, we have
\begin{equation*}
\mathbb{E}_{\gamma} | [RVR]_{\mu k}\overline{[(RV)^2R]}_{\nu k}| \leq N^{-3/2+C\epsilon_0}\Psi(z)^{3-s}.  
\end{equation*}
The other cases $4 \leq m \leq 8$ can be handled similarly.  Hence, we conclude our proof. 
\end{proof}

Lemma \ref{greenfunction_comp} follows from the following result. Recall (\ref{reduction}), denote 
\begin{equation*}
\Delta x(E):=x^S(E)-x^R(E), \ \Delta y(E):=y^S(E)-y^R(E).
\end{equation*}
\begin{lem} \label{keykeykeykey}
For any fixed $\mu, \nu, \gamma$, there exists a random variable $A$, which depends on the randomness only through $O$ and the first two moments of $X^G$,  such that
\begin{equation} \label{finalbound}
\mathbb{E} \theta [\int_I x^S q(y^S)dE]-\mathbb{E}\theta [\int_I x^R q(y^R)dE]=A+o(N^{-2+t}),
\end{equation}
where $t:=|\{\mu, \nu\} \cap \{\mu_1\}|$ and $t=0, 1$ counts if there is $\mu, \ \nu$ equals to $\mu_1.$ 
\end{lem}
Before proving Lemma \ref{keykeykeykey}, we firstly show how Lemma \ref{keykeykeykey} implies Lemma \ref{greenfunction_comp}.
\begin{proof}[Proof of Lemma \ref{greenfunction_comp}] It is easy to check that Lemma \ref{keykeykeykey} still holds true when we replace $S$ with $T$. Note in (\ref{finalgreenfunction}),  there are $O(N)$ terms when $t=1$ and $O(N^2)$ terms when $t=0$. By  (\ref{finalbound}), we have
\begin{equation*}
\mathbb{E}[\theta(\int_I x^G q(y^G)dE)-\theta(\int_I x^V q(y^V)dE)]=o(1),
\end{equation*} 
where we use the assumption that the first two moments of $X^V$ are the same with $X^G.$
Combine with (\ref{greenfunctionreduced1}), we conclude the proof.
\end{proof}
Finally we will follow the approach of \cite[Lemma 3.6]{KY} to finish the proof of Lemma \ref{keykeykeykey}. A key observation is that when $s=0,$ we will have a smaller bound but the total number of such terms are $O(N)$ for $x(E)$ and $O(N^2)$ for $y(E).$  And when $s=1,$ we have a larger bound but the number of such terms are $O(1).$   We need to analyze the items with $s=0,1$ separately. 
\begin{proof}[Proof of Lemma \ref{keykeykeykey}]
Condition on the variable $s=0,1,$ we introduce the following decomposition 
\begin{equation*}
x_s(E):= \frac{N \eta}{\pi} \sum_{k=M+1, \ \text{and} \  \neq \mu, \nu}^{M+N} X_{\mu\nu,k}(E+i \eta) \mathbf{1}(s=\mathbf{1}\left((\{ \mu, \nu\} \cap \{ \mu_1 \} \neq \emptyset) \cup (\{ k=\mu_1\})\right)),
\end{equation*}
\begin{equation*}
y_s(E):=\frac{\tilde{\eta}}{\pi} \int^{E_U}_{E^{-}} \sum_k \sum_{\beta \neq k}X_{\beta \beta,k}(E+i\tilde{\eta})dE \mathbf{1}(s=\mathbf{1}((\{ \beta= \mu_1 \}) \cup (\{ k=\mu_1\}))).
\end{equation*}
$\Delta x_s, \Delta y_s$ can be defined in the same fashion. Similar to the discussion of (\ref{random1}), for any $E$-dependent variable $f \equiv f(E)$ independent of the $(i,\mu_1)$-th entry of $X^{G},$ there exist two random variables $A_2, A_3$, which depend on the randomness only through $O$, $f$ and the first two moments of $X^G_{i \mu_1}$, for any event $\Omega,$  with $1-N^{-D_1}$ probability,  we have
\begin{equation} \label{random2}
\left | \int_I \mathbb{E}_{\gamma} \Delta x_s(E)f(E)dE-A_2 \right | \mathbf{1}(\Omega) \leq \vert \vert f \mathbf{1}(\Omega) \vert \vert_{\infty} N^{-11/6+C\epsilon_0} N^{-2s/3+t},
\end{equation}
\begin{equation} \label{random4}
\left |  \mathbb{E}_{\gamma}  \Delta y_s(E)-A_3 \right | \leq N^{-11/6+C\epsilon_0}N^{-2s/3}.
\end{equation}
In our application, $f$ is usually a function of the entries of $R$ (recall $R$ is independent of $V$).   Next, we use
\begin{equation} \label{beforetaylor}
\theta [\int_I x^S q(y^S)dE]=\theta[\int_I (x^R+\Delta x_0+\Delta x_1)q(y^R+\Delta y_0+ \Delta y_1)dE].
\end{equation}
By (\ref{resolvent}), (\ref{combinations}) and (\ref{combinations1}), it is easy to check that, with $1-N^{-D_1}$ probability, we have
\begin{equation} \label{elementbound1}
\int_I |\Delta x_s(E)|dE \leq N^{-5/6+C\epsilon_0}N^{-2s/3+t} , \ | \Delta y_s(E) | \leq N^{-5/6+C\epsilon_0}N^{-2s/3},
\end{equation}
\begin{equation} \label{elementbound2}
\int_I |x(E)| dE \leq N^{C\epsilon_0}, \ |y(E)| \leq N^{C\epsilon_0}.
\end{equation}
By (\ref{beforetaylor}) and (\ref{elementbound1}),  with $1-N^{-D_1}$ probability, we have 
\begin{align*}
\theta [\int_I x^S q(y^S)dE]= \theta[\int_I x^S(q(y^R)+q^{\prime}(y^R)(\Delta y_0+\Delta y_1)+q^{\prime \prime}(y^R)(\Delta y_0)^2)dE]+o(N^{-2}). 
\end{align*}
Similarly, we have (see (3.44) of \cite{KY})
\begin{align} \label{lasttwo}
& \theta  [\int_I x^S q(y^S)dE]- \theta  [\int_I x^R q(y^R)dE]=\theta^{\prime} [\int_I x^R q(y^R)dE] \nonumber \\
& \times [\int_I \left( (\Delta x_0 + \Delta x_1)q(y^R)+x^Rq^{\prime}(y^R)(\Delta y_0 + \Delta y_1)+\Delta x_0 q^{\prime}(y^R)\Delta y_0+x^R q^{\prime \prime}(y^R)(\Delta y_0)^2 \right) dE] \nonumber \\
&+ \frac{1}{2}\theta^{\prime \prime}[\int_I x^R q(y^R)dE][\int_I (\Delta x_0q(y^R)+x^R q^{\prime}(y^R)\Delta y_0) dE]^2+o(N^{-2+t}).
\end{align}
Now we start dealing with the individual terms on the right-hand side of (\ref{lasttwo}). Firstly, we consider the terms containing $\Delta x_1, \ \Delta y_1$. Similar to (\ref{random1}), we can find a random variable $A_4,$ which depends on randomness only through $O$ and the first two moments of $X^G_{i \mu_1},$ such that with $1-N^{-D_1}$ probability,
\begin{equation*}
\left|\mathbb{E}_{\gamma} \int_I (\Delta x_1 q(y^R)+x^R q^{\prime}(y^R)\Delta y_1) dE-A_4 \right|=o(N^{-2+t}).
\end{equation*}
Hence, we only need to focus on $\Delta x_0, \ \Delta y_0.$ We first observe that
\begin{equation*}
\Delta x_0(E)=\mathbf{1}(t=0) \frac{N \eta}{\pi} \sum_{k \neq \mu,  \nu, \mu_1 } \Delta X_{\mu \nu, k}(z),
\end{equation*}
\begin{equation*}
\Delta y_0(E)= \frac{\tilde{\eta}}{\pi} \int^{E_U}_{E^{-}} \sum_{k \neq \mu_1} \sum_{\beta \neq k, \mu_1} \Delta X_{\beta \beta, k}(E+i\tilde{\eta}) dE.
\end{equation*}
Denote $\Delta x_0^{(k)}(E)$ by the summations of the terms in $\Delta x_0(E)$ containing $k$ items of $X^G_{i \mu_1}$. By (\ref{edgebound}), (\ref{resolvent}) and (\ref{combinations}), it is easy to check that with $1-N^{-D_1}$ probability, 
\begin{equation} \label{bbbbdddd}
|\Delta x_0^{(3)}| \leq N^{-7/6+C\epsilon_0}, \ |\Delta y_0^{(3)}| \leq N^{-11/6+C\epsilon_0}. 
\end{equation}
We now decompose $\Delta X_{\mu \nu, k}$ into three parts indexed by the number of $X^G_{i \mu_1}$ they contain. By (\ref{edgebound}), (\ref{combinations}), (\ref{combinations1}) and (\ref{bbbbdddd}), with $1-N^{-D_1}$ probability,  we have
\begin{equation} \label{deltadecop1}
\Delta X_{\mu \nu, k}= \Delta X^{(1)}_{\mu \nu,k}+ \Delta X^{(2)}_{\mu \nu,k}+  \Delta X^{(3)}_{\mu \nu,k}+ O(N^{-3+C\epsilon_0}),
\end{equation}
\begin{equation}\label{deltadecop2}
\Delta x_0=\Delta x_0^{(1)}+\Delta x_0^{(2)}+ \Delta x_0^{(3)}+O(N^{-5/3+C \epsilon_0}),
\end{equation}
\begin{equation}\label{deltadecop3}
\Delta y_0=\Delta y_0^{(1)}+\Delta y_0^{(2)}+ \Delta y_0^{(3)}+O(N^{-7/3+C \epsilon_0}).
\end{equation}
Inserting (\ref{deltadecop2}) and (\ref{deltadecop3}) into (\ref{lasttwo}), similar to the discussion of (\ref{random1}), we can find a random variable $A_5$ depending on the randomness only through $O$ and the first two moments of $X^G_{i \mu_1},$ such that  with $1-N^{-D_1}$ probability,
\begin{align} \label{lasttwotwo}
& \mathbb{E}_{\gamma} \theta [\int_I x^S q(y^S)dE]-\mathbb{E}_{\gamma} \theta [\int_I x^R q(y^R)dE] \nonumber \\
& = \mathbb{E}_{\gamma} \theta^{\prime}[\int_I x^R q(y^R)dE][\int_I \Delta x_0^{(3)} q(y^R)+x^R q^{\prime}(y^R)\Delta y_0^{(3)}dE]+ A_4+A_5+o(N^{-2+t}).
\end{align}

\noindent Lemma \ref{keykeykeykey} will be proved if we can show
\begin{equation} \label{llllll}
\mathbb{E}\theta^{\prime}[\int_I x^R q(y^R)dE][\int_I \Delta x_0^{(3)} q(y^R)+x^R q^{\prime}(y^R)\Delta y_0^{(3)}dE]=o(N^{-2}).
\end{equation}
Due to the similarity, we shall prove
\begin{equation*}
\mathbb{E}\theta^{\prime}[\int_I x^R q(y^R)dE][\int_I \Delta x_0^{(3)} q(y^R)dE]=o(N^{-2}),
\end{equation*}
the other term follows. By (\ref{defn_theta1d}) and (\ref{elementbound2}), with $1-N^{-D_1}$ probability, we have $|B^R|:=\left | \theta^{\prime}[\int_I x^R q(y^R)dE] \right |\leq N^{C\epsilon_0}.$ Similar to (\ref{combi1}), $\Delta x_0^{(3)}$ is a finite sum of terms of the form 
\begin{equation}\label{deltax3expanson}
\mathbf{1}(t=0) N \eta \sum_{k \neq \mu, \nu, \mu_1 } R_{\mu k} (\sigma_i)^{3/2} (X^G_{i \mu_1})^3 \overline{z^{3/2} R_{\nu a_1} R_{b_1 a_2} R_{b_2 a_3} R_{b_3 k}}.
\end{equation}
Inserting (\ref{deltax3expanson}) into $\int_I \Delta x_0^{(3)} q(y^R)dE$,  for some constant $C>0$, we have
\begin{align} \label{lastlastlast}
\left | \mathbb{E} \theta^{\prime}[\int_I x^R q(y^R)dE][\int_I \Delta x_0^{(3)} q(y^R)dE] \right | \leq N^{-5/6+C\epsilon_0}\max_{k \neq \mu, \nu, \mu_1} \sup_{E \in I}\left | \mathbb{E}B^R R_{\mu k} \overline{R_{\nu \mu_1}R_{i k}} q(y^R) \right |+o(N^{-2}).
\end{align}
Again by (\ref{resolvent}), (\ref{combinations}) and (\ref{combinations1}), it is easy to check that with $1-N^{-D_1}$ probability,  for some constant $C>0$, we have  
\begin{equation*}
|R_{\mu k} \overline{R_{\nu \mu_1} R_{i k}} B^R q(y^R)-S_{\mu k}\overline{S_{\nu \mu_1} S_{i k}} B^S q(y^S) | \leq N^{-4/3+C \epsilon_0}.
\end{equation*}
Therefore, if we can show
\begin{equation} \label{11212121}
|\mathbb{E}S_{\mu k}\overline{S_{\nu \mu_1} S_{i k}} B^S q(y^S)| \leq N^{-4/3+C \epsilon_0},
\end{equation}
then by (\ref{lastlastlast}), we finish proving (\ref{llllll}). The rest leaves to prove (\ref{11212121}). Recall Definition \ref{minors} and (\ref{Hgamma}), by \cite[Lemma 3.2 and 3.3]{XYY}(or \cite[Lemma A.2]{DY}), we have the following resolvent identities, 
\begin{equation}\label{resolvent1}
S^{(\mu_1)}_{\mu \nu}=S_{\mu \nu}-\frac{S_{\mu \mu_1}S_{\mu_1 \nu}}{S_{\mu_1 \mu_1}}, \ \mu, \nu \neq \mu_1, 
\end{equation}
\begin{equation} \label{resolvent2}
S_{\mu \nu}=zS_{\mu \mu} S_{\nu \nu}^{(\mu)}(Y_{\gamma-1}^* S^{(\mu \nu)} Y_{\gamma-1})_{\mu \nu}, \ \mu \neq \nu.
\end{equation}
By (\ref{combinations}), (\ref{combinations1}) and (\ref{resolvent1}),  it is easy to check that (see (3.72) of \cite{KY}),
\begin{equation}\label{tri1}
|S_{\mu k}\overline{S_{\nu \mu_1} S_{i k}} B^S q(y^S)-S^{(\mu_1)}_{\mu k}\overline{S_{\nu \mu_1} S^{(\mu_1)}_{i k}} (B^S)^{(\mu_1)} q((y^S)^{(\mu_1)})| \leq N^{-4/3+C \epsilon_0}.
\end{equation}
Moreover, by (3.73) of \cite{KY},  we have
\begin{equation}\label{131414141}
S^{(\mu_1)}_{\mu k}\overline{S_{\nu \mu_1} S^{(\mu_1)}_{i k}} (B^S)^{(\mu_1)} q((y^S)^{(\mu_1)})=(S_{\mu k} \overline{S_{i k}}B^S q(y^S))^{(\mu_1)} \overline{S}_{\nu \mu_1}.
\end{equation}
As $t=0,$ by (\ref{resolvent2}), we have
\begin{equation} \label{fffff}
S_{\nu \mu_1}= zm(z) S_{\mu_1 \mu_1}^{(\nu)} \sum_{p,q} S_{pq}^{(\nu \mu_1)} (Y_{\gamma-1}^*)_{\nu p} (Y_{\gamma-1})_{q\mu_1}+z(S_{\nu \nu}-m(z)) S_{\mu_1 \mu_1}^{(\nu)} \sum_{p,q} S_{pq}^{(\nu \mu_1)} (Y_{\gamma-1}^*)_{\nu p} (Y_{\gamma-1})_{q\mu_1}.
\end{equation}
The conditional expectation $\mathbb{E}_{\gamma}$ applied to the first term of (\ref{fffff}) vanishes; hence its contribution to the expectation of (\ref{131414141}) will vanish. By (\ref{thm_anisotropic_eq12}), with $1-N^{-D_1}$ probability, we have 
\begin{equation} \label{deq}
|S_{\nu \nu}-m(z)| \leq N^{-1/3+C\epsilon_0}.
\end{equation}
By the large deviation bound \cite[Lemma 3.6]{XYY}, with $1-N^{-D_1}$ probability, we have
\begin{equation} \label{largederivation}
\left|\sum_{p,q} S_{pq}^{(\nu \mu_1)} (Y_{\gamma-1}^*)_{\nu p} (Y_{\gamma-1})_{q\mu_1} \right| \leq N^{\epsilon_1} \frac{(\sum_{p,q} |S_{pq}^{(\nu \mu_1)}|^2)^{1/2}}{N}.
\end{equation}
By (\ref{thm_anisotropic_eq12})  and (\ref{largederivation}), with $1-N^{-D_1}$ probability, we have
\begin{equation}\label{ndeq}
\left|\sum_{p,q} S_{pq}^{(\nu \mu_1)} (Y_{\gamma-1}^*)_{\nu p} (Y_{\gamma-1})_{q\mu_1} \right| \leq N^{-1/3+C\epsilon_0}.
\end{equation}
Therefore, inserting (\ref{deq}) and (\ref{ndeq}) into (\ref{131414141}), by (\ref{thm_anisotropic_eq12}), we have
\begin{equation*}
|\mathbb{E}S^{(\mu_1)}_{\mu k}\overline{S_{\nu \mu_1} S^{(\mu_1)}_{i k}} (B^S)^{(\mu_1)} q((y^S)^{(\mu_1)})| \leq N^{-4/3+C \epsilon_0}.
\end{equation*}
Combine with (\ref{tri1}),  we conclude our proof.
\end{proof}
It is clear that our proof can be extended to the left singular vectors. For the proof of Theorem \ref{thm_edge}, the only difference is to use mean value theorem in $\mathbb{R}^2$ whenever it is needed. Moreover,  for the proof of Theorem \ref{thm_edge_sev}, we need to use $n$ intervals defined by
$$ I_i:=[a_{2k_i-1}-N^{-2/3+\epsilon}, a_{2k_i-1}+N^{-2/3+\epsilon}], \  i=1,2,\cdots,n. $$

\subsection{Extension to singular values}

In this section,  we will discuss how the arguments of Section \ref{greenfunction_comp_edge} can be applied to  the general function $\theta$ defined in (\ref{defn_theta3d}) containing singular values.  We mainly focus on discussing the proofs of Corollary \ref{corollaryedge}.

 On one hand, similar to Lemma \ref{truncation}, we can write the singular values in terms of an integral of smooth functions of Green functions. Using the comparison argument with $\theta \in \mathbb{R}^3$ and mean value theorem in $\mathbb{R}^3$, we can conclude our proof.  Similar discussions and results have been derived in \cite[Corollary 6.2 and Theorem 6.3]{EYY}. For completeness of the paper,  we basically follow the strategy of \cite[Section 4]{KY} to prove Corollary \ref{corollaryedge}.  The basic idea is to write the function $\theta$ in terms of Green functions by using integration by parts. We mainly look at the right edge of the $k$-th bulk component.
\begin{proof}[Proof of Corollary \ref{corollaryedge}]  Denote $F^V$ be the law of $\lambda_{\alpha^{\prime}}$,  consider a smooth function $\theta: \mathbb{R} \rightarrow \mathbb{R},$  for  $\delta$ defined in Lemma \ref{step1}, when $l \leq N_k^{\delta}$,  by (\ref{edgeeigenvalue}) and (\ref{thm_rigidity_equ}), it is easy to check that
\begin{equation} \label{inter0}
\mathbb{E}^V \theta(\frac{N^{2/3}}{\varpi}(\lambda_{\alpha^{\prime}}-a_{2k-1}))=\int_I \theta(\frac{N^{2/3}}{\varpi}(E-a_{2k-1}))dF^V(E)+O(N^{-D_1}), 
\end{equation}
where $\varpi:=\varpi_{2k-1}$ and $I$ is defined in (\ref{defn_i_1}). Using integration by parts on (\ref{inter0}), we have
\begin{equation} \label{inter1}
[\mathbb{E}^V-\mathbb{E}^G] \theta(\frac{N^{2/3}}{\varpi}(\lambda_{\alpha^{\prime}}-a_{2k-1})) =-[\mathbb{E}^V-\mathbb{E}^G] \int_I \frac{N^{2/3}}{\varpi} \theta^{\prime}(\frac{N^{2/3}}{\varpi}(E-a_{2k-1})) \mathbf{1}(\lambda_{\alpha^{\prime}}\leq E)dE+O(N^{-D_1}),
\end{equation}
where we use (\ref{edgeeigenvalue}) and  (\ref{thm_rigidity_equ}). Similar to  (\ref{defn_q}),  recall (\ref{defnalphaprime1}),  choose a smooth nonincreasing function $f_{l}$ that vanishes on the interval $[l+2/3, \infty)$ and is equal to 1 on the interval $(-\infty, l+1/3].$  Recall that $E_U=a_{2k-1}+2N^{-2/3+\epsilon}$  and $\mathcal{N}(E, E_U)$ denotes the number of eigenvalues of $Q_1$ locate in the interval $[E, E_U],$ by (\ref{inter1}),  we have
\begin{align*}
[\mathbb{E}^V-\mathbb{E}^G]\theta(\frac{N^{2/3}}{\varpi}(\lambda_{\alpha^{\prime}}-a_{2k-1}))=-[\mathbb{E}^V-\mathbb{E}^G] \int_I \frac{N^{2/3}}{\varpi} \theta^{\prime}(\frac{N^{2/3}}{\varpi}(E-a_{2k-1})) f_{l}(\mathcal{N}(E, E_U))dE+O(N^{-D_1}).
\end{align*}
Recall $\tilde{\eta}=N^{-2/3-9\epsilon_0},$ similar to the discussion of  (\ref{lem_deltaestimate}),  with $1-N^{-D_1}$ probability, we have 
\begin{equation*}
N^{2/3} \int_I  \left |  \operatorname{Tr} (\mathbf{1}_{[E, E_U]}*\vartheta_{\tilde{\eta}}(Q_1)))-\operatorname{Tr} (\mathbf{1}_{[E, E_U]}(Q_1)) \right | dE \leq N^{-\epsilon_0}.
\end{equation*}
This yields that 
\begin{equation*}
[\mathbb{E}^V-\mathbb{E}^G]\theta(\frac{N^{2/3}}{\varpi}(\lambda_{\alpha^{\prime}}-a_{2k-1}))=-[\mathbb{E}^V-\mathbb{E}^G] \int_I \frac{N^{2/3}}{\varpi} \theta^{\prime}(\frac{N^{2/3}}{\varpi}(E-a_{2k-1})) f_{l}(\operatorname{Tr}(\mathbf{1}_{[E,E_U]}*\vartheta_{\tilde{\eta}}(Q_1)))dE+O(N^{-D_1}).
\end{equation*}
\noindent By integration by parts, we have  
\begin{equation*} \label{inter2}
[\mathbb{E}^V-\mathbb{E}^G] \theta (\frac{N^{2/3}}{\varpi}(\lambda_{\alpha^{\prime}}-a_{2k-1}))=\frac{N}{\pi}[\mathbb{E}^V-\mathbb{E}^G]   \int_I \theta (\frac{N^{2/3}}{\varpi}(\lambda_{\alpha^{\prime}}-a_{2k-1})) f_{l}^{\prime}(\operatorname{Tr}(\mathbf{1}_{[E, E_U]}*\vartheta_{\tilde{\eta}}(Q_1))) \operatorname{Im}m_2(E+i\tilde{\eta})dE+o(1),
\end{equation*}
where we use (\ref{keyformula}).  Now we extend $\theta$ to the general case defined in (\ref{defn_theta3d}).  By Theorem \ref{thm_edge}, it is easy to check that
\begin{align} \label{finalproofexpression}
& [\mathbb{E}^V-\mathbb{E}^G]\theta(\frac{N^{2/3}}{\varpi}(\lambda_{\alpha^{\prime}}-a_{2k-1}),  N \xi_{\alpha^{\prime}}(i) \xi_{\alpha^{\prime}}(j), N \zeta_{\alpha^{\prime}}(\mu) \zeta_{\alpha^{\prime}}(\nu)) \nonumber \\
= &  \frac{1}{\pi}[\mathbb{E}^V-\mathbb{E}^G] \int_I \theta (\frac{N^{2/3}}{\varpi}(\lambda_{\alpha^{\prime}}-a_{2k-1}), \phi_{\alpha^{\prime}}, \varphi_{\alpha^{\prime}}) f_{l}^{\prime}(\operatorname{Tr}(\mathbf{1}_{[E, E_U]}*\theta_{\tilde{\eta}}(Q_1)))N \operatorname{Im}m_2(E+i\tilde{\eta})dE+o(1), 
\end{align}
where we introduce the shorthand notations
\begin{equation*}
\phi_{\alpha^{\prime}}=\frac{N}{\pi} \int_I \tilde{G}_{ij}(\tilde{E}+i\eta)q_1 [\operatorname{Tr}( \mathbf{1}_{[\tilde{E}^{-}, E_U]}* \vartheta_{\tilde{\eta}}(Q_1))] d\tilde{E}, 
\end{equation*}
\begin{equation*}
\varphi_{\alpha^{\prime}}=\frac{N}{\pi} \int_I \tilde{G}_{\mu \nu}(\tilde{E}+i\eta)q_2 [\operatorname{Tr}( \mathbf{1}_{[\tilde{E}^{-}, E_U]}* \vartheta_{\tilde{\eta}}(Q_1))] d\tilde{E}.
\end{equation*}
and $q_1, \ q_2$ are functions defined in (\ref{defn_q}).  Therefore, the randomness on the right-hand side of (\ref{finalproofexpression}) is expressed in terms of Green functions. Hence, we can apply the Green function comparison argument to (\ref{finalproofexpression}) as in Section \ref{greenfunction_comp_edge}.  The complications are notational and we will not reproduce the detail here.
\end{proof}
Finally, the proofs of Corollary \ref{coroedgesev} are very similar to that of Corollary \ref{corollaryedge} except we will  use $n$ different intervals and a  multidimensional integral. We will not reproduce the detail here.

\section{Singular vectors in the bulks} \label{bulk}

In this section, we will prove the bulk universality Theorem \ref{thm_bulk} and \ref{thm_bulksev}.  Our key ingredients Lemma \ref{lem_anisotropic}, \ref{de_local}  and Corollary \ref{them_specific_bounds}  are proved for $N^{-1+\tau} \leq \eta \leq \tau^{-1}$ (recall (\ref{DOMAIN1})). In the bulks, recall Lemma \ref{ridgitybulk}, the eigenvalue spacing is of order $N^{-1}$. The following lemma extends the above controls for a small spectral scale all the way down to the real axis. The proof relies on Corollary \ref{them_specific_bounds} and  the detail can be found in \cite[Lemma 5.1]{KY}.
 
\begin{lem} \label{anisotropicsmall} Recall (\ref{domainbulk}), for $z \in \mathbf{D}_k^b \ \text{with} \ 0<\eta \leq \tau^{-1},$ when $N$ is large enough, with $1-N^{-D_1}$ probability,  we have  
\begin{equation} \label{bulkbound}
\max_{\mu, \nu} |G_{\mu \nu}-\delta_{\mu \nu} m(z)\vert \leq N^{\epsilon_1} \Psi(z).
\end{equation}
\end{lem}

Once Lemma \ref{anisotropicsmall} is established, Lemma \ref{ridgitybulk} and \ref{de_local} will follow. Next we follow the basic proof strategy for Theorem \ref{thm_edge} but use different spectral window size.  Again, we will only provide the proof for the following Lemma \ref{lem_onedimensionalcasebulk}, which establishes the universality for the distribution of $\zeta_{\alpha^{\prime}}(\mu) \zeta_{\alpha^{\prime}}(\nu)$ in detail.  To the end of this section, we always use the scale parameter
\begin{equation} \label{bulketa}
\eta =N ^{-1-\epsilon_0}, \ \epsilon_0>\epsilon_1 \ \text{is a small constant}.
\end{equation}
Therefore,  the following bounds hold with $1-N^{-D_1}$ probability 
\begin{equation} \label{bulkboundsall}
\max_{\mu} |G_{\mu \mu}(z)|\leq N^{2 \epsilon_0}, \ \max_{\mu \neq \nu} |G_{\mu \nu}(z)|\leq N^{2\epsilon_0}, \  \max_{\mu, s} \vert \zeta_{\mu}(s) \vert^2 \leq N^{-1+\epsilon_0}.
\end{equation}
The following lemma states the bulk universality for  $\zeta_{\alpha^{\prime}}(\mu) \zeta_{\alpha^{\prime}}(\nu)$.
\begin{lem}\label{lem_onedimensionalcasebulk} For $Q_V=\Sigma^{1/2}X_VX_V^* \Sigma^{1/2}$ satisfying Assumption \ref{assu_main},  assuming that the third and fourth moments of $X_V$ agree with those of $X_G$ and  considering the $k$-th bulk component, $k=1,2,\cdots,p$ and  $l$ defined in (\ref{defnalphaprime1}) or (\ref{defnalphaprime2}) , under Assumption \ref{defnregularity} and \ref{assum_nozero}, for  any choices of indices $\mu, \nu \in \mathcal{I}_2$, there exists a small $\delta \in (0,1),$ when $\delta N_k \leq l \leq (1-\delta)N_k,$  we have
\begin{equation*}
\lim_{N \rightarrow \infty} [\mathbb{E}^V-\mathbb{E}^G] \theta( N \zeta_{\alpha^{\prime}}(\mu) \zeta_{\alpha^{\prime}}(\nu))=0,
\end{equation*}
where $ \theta$ is a smooth function in $\mathbb{R}$ that satisfies
\begin{equation} \label{defn_theta1bd}
\vert \theta^{(5)}(x) \vert \leq C_1 (1+ \vert x \vert)^{C_1},  \ \text{with some constant} \ C_1>0.  
\end{equation}
\end{lem}

\begin{proof}
The proof strategy  is very similar to that of Lemma \ref{lem_onedimensionalcase}. Our first step is an analogue of Lemma \ref{step1}. The proof is quite similar (actually easier as the window size is much smaller). We omit further detail.
\begin{lem} \label{step1bulk} Under the assumptions of Lemma \ref{lem_onedimensionalcasebulk},  there exists a $0<\delta<1$,  we have
\begin{equation} \label{lem_step1_eq}
\lim_{N \rightarrow \infty} \max_{\delta N_k \leq l \leq (1-\delta)N_k } \max_{\mu, \nu} \left \vert \mathbb{E}^V \theta(N \zeta_{\alpha^{\prime}}(\mu)\zeta_{\alpha^{\prime}}(\nu))-\mathbb{E}^V \theta[\frac{N}{\pi}\int_{I} \tilde{G}_{\mu \nu}(z) \mathcal{X}(E) dE] \right\vert =0,
\end{equation}
where $\mathcal{X}(E)$ is defined in (\ref{defn_x(e)}) and for $\epsilon$ satisfying (\ref{defn_epsion}), $I$ is denoted as 
\begin{equation}\label{defn_i_1bulk}
I:=[\gamma_{\alpha^{\prime}}-N^{-1+\epsilon}, \ \gamma_{\alpha^{\prime}}+N^{-1+\epsilon}]. 
\end{equation}
\end{lem}

Next we will express the indicator function in (\ref{lem_step1_eq}) using Green functions. Recall (\ref{keywindowlength}),  a key observation  is that the size of $[E^{-}, E_U]$ is of order $N^{-2/3}$ due to  (\ref{edgeeta}). As we now use (\ref{bulketa}) and (\ref{defn_i_1bulk}) in the bulks, the size here is of order 1. So we cannot use the delta approximation function to estimate $\mathcal{X}(E)$. Instead, we will use Helffer-Sj{\" o}strand functional calculus. This has been used many times when the window size $\eta$ takes the form of (\ref{bulketa}), for example in the proofs of rigidity of eigenvalues in \cite{DY, EYY,PY1}.

For any $0< E_1, \ E_2\leq \tau^{-1}$, denote 
$f(\lambda)\equiv f_{E_1, E_2,\eta_d}(\lambda)$ be the characteristic function of $[E_1, E_2]$ smoothed on the scale 
\begin{equation}\label{defn_ed}
\eta_d:=N^{-1-d\epsilon_0},  \ d >2,
\end{equation}
where $f=1,$ when $\lambda \in [E_1, E_2]$ and $f=0$ when $\lambda \in \mathbb{R} \backslash [E_1-\eta_d, E_2+\eta_d ]$, and 
\begin{equation} \label{functionalcalculusbound}
|f^{\prime}| \leq C \eta_d^{-1}, \ |f^{\prime \prime}| \leq C \eta_d^{-2},
\end{equation}
for some constant $C>0.$ By (B.12) of \cite{ERSY}, denote $f_E \equiv f_{E^{-}, E_U,  \eta_d},$  we have 
\begin{equation} \label{functinoalcalculus}
f_{E}(\lambda)=\frac{1}{2\pi} \int_{\mathbb{R}^2} \frac{i\sigma f^{\prime \prime}_E(e)\chi(\sigma)+if_E(e)\chi^{\prime}({\sigma})-\sigma f^{\prime}_E(e) \chi^{\prime}(\sigma)}{\lambda-e-i\sigma}de d\sigma,
\end{equation} 
where $\chi(y)$ is a smooth cutoff function with support  $[-1,1]$ and $\chi(y)=1$ for $|y|\leq \frac{1}{2}$ with bounded derivatives. Using a similar argument to Lemma \ref{truncation},  we have the following result. 
\begin{lem} \label{truncationbulk} Recall the smooth cutoff function  $q$ defined in (\ref{defn_q}),  under the assumptions of Lemma \ref{step1bulk}, there exists a $0<\delta<1,$ such that
\begin{equation} \label{lem_step1_eqbulk}
\lim_{N \rightarrow \infty} \max_{\delta N_k \leq l \leq (1-\delta)N_k }  \max_{\mu, \nu} \left \vert \mathbb{E}^V \theta[\frac{N}{\pi}\int_{I} \tilde{G}_{\mu \nu}(z) \mathcal{X}(E)] dE-\mathbb{E}^V \theta[\frac{N}{\pi} \int_{I} \tilde{G}_{\mu \nu}(z)q(\operatorname{Tr}f_E(Q_1))]dE \right\vert =0.
\end{equation}
\end{lem}
\begin{proof}
It is easy to check that with $1-N^{-D_1}$ probability,  (\ref{part_estimate}) still holds true. Therefore, it remains to prove the following result
\begin{equation}\label{reduceproof}
\mathbb{E}^V \theta [\frac{N}{\pi} \int_I \tilde{G}_{\mu \nu}(E+i\eta)q(\operatorname{Tr}\mathcal{X}_E (Q_1))]-\mathbb{E}^V\theta[\frac{N}{\pi} \int_I \tilde{G}_{\mu \nu}(E+i\eta)q(\operatorname{Tr}f_E(Q_1))dE]=o(1).
\end{equation}
We first observe that  for  any $x \in \mathbb{R},$ we have
\begin{equation*}
| \mathcal{X}_E(x)-f_E(x) |=
\begin{cases}
0, & x \in [E^{-},E_U] \cup (-\infty, E^- -\eta_d) \cup (E_U+\eta_d, +\infty); \\
|f_E(x)|, & x \in [E^{-}-\eta_d, E^{-}) \cup (E_U, E_U+\eta_d].
\end{cases}
\end{equation*}
Therefore, we have
\begin{equation*}
| \operatorname{Tr}\mathcal{X}_E(Q_1)-\operatorname{Tr} f_E(Q_1)| \leq \max_{x} |f_E(x)| \left ( \mathcal{N}(E^-- \eta_d, E^-)+\mathcal{N}(E_U, E_U+\eta_d) \right).
\end{equation*}
By Lemma \ref{ridgitybulk}, the definition of $\eta_d$ and a similar argument to (\ref{truncation_estimate_eq}), we  can finish the proof of (\ref{reduceproof}). 
\end{proof}
\noindent Finally, we apply the Green function comparison argument, where we will follow the basic approach of Section \ref{greenfunction_comp_edge} and \cite[Section 5]{KY}. The key difference is that we will use (\ref{bulketa}) and (\ref{bulkboundsall}).
\begin{lem}\label{comparisonbulk} Under the assumptions of Lemma \ref{truncationbulk},  there exists a $0<\delta<1$,  we have
\begin{equation} \label{bulkcomparisoneq}
\lim_{N \rightarrow \infty} \max_{\delta N_k \leq l \leq (1-\delta)N_k }   \max_{\mu, \nu} [\mathbb{E}^V-\mathbb{E}^G]\theta[\frac{N}{\pi} \int_I \tilde{G}_{\mu \nu}(E+i\eta)q(\operatorname{Tr}f_E(Q_1))dE]=0.
\end{equation}
\end{lem}
\begin{proof}
Recall (\ref{functinoalcalculus}), by (\ref{EEEE}), we have
\begin{equation} \label{bulkdecompose1}
\operatorname{Tr} f_{E}(Q_1)=\frac{N}{2 \pi} \int_{\mathbb{R}^2}(i \sigma f^{\prime \prime}_E(e) \chi(\sigma)+if_E(e)\chi^{\prime}(\sigma)-\sigma f_{E}^{\prime}(e) \chi^{\prime}(\sigma))m_2(e+i\sigma)de d\sigma.
\end{equation}
Denote $\tilde{\eta}_d:=N^{-1-(d+1)\epsilon_0}$, we can decompose the right-hand side of (\ref{bulkdecompose1}) by
\begin{align*}
\operatorname{Tr} f_{E}(Q_1)= \frac{N}{2 \pi} \int \int_{\mathbb{R}^2}(if_E(e)\chi^{\prime}(\sigma)-\sigma f^{\prime}_E(e)\chi^{\prime}(\sigma))m_2(e+i\sigma)ded\sigma+ \frac{i N}{2 \pi} \int_{|\sigma|>\tilde{\eta}_d} \sigma\chi(\sigma)   \int f_E^{\prime \prime}(e)m_2(e+i\sigma) d\sigma de   \nonumber \\
+\frac{i N}{2 \pi} \int_{-\tilde{\eta}_d}^{\tilde{\eta}_d}\sigma\chi(\sigma)   \int f_E^{\prime \prime}(e)m_2(e+i\sigma) d\sigma  de  . 
\end{align*}
By (\ref{bulkboundsall}) and  (\ref{functionalcalculusbound}),  for some constant $C>0,$ with $1-N^{-D_1}$ probability, we have
\begin{equation} \label{bulkremainder1}
\left | \frac{i N}{2 \pi} \int_{-\tilde{\eta}_d}^{\tilde{\eta}_d}\sigma\chi(\sigma)   \int f_E^{\prime \prime}(e)m_2(e+i\sigma) d\sigma de  \right | \leq N^{-C\epsilon_0}.
\end{equation}
Recall (\ref{eigenvectorrepresentation}) and (\ref{reduction}), similar to Lemma \ref{dropdiagonal}, we firstly drop the diagonal terms. By (\ref{bulkbound}), with $1-N^{-D_1}$ probability,  we have (recall (\ref{keyformula0}))
\begin{equation*}
\int_I \left  | \frac{N}{\pi} \tilde{G}_{\mu \nu}(E+i\eta)-x(E) \right | dE \leq  N^{-1+C\epsilon_0},
\end{equation*}
for some constant $C>0.$ Hence, by mean value theorem, we only need to prove
\begin{equation}
\lim_{N \rightarrow \infty} \max_{\delta N_k \leq l \leq (1-\delta)N_k }   \max_{\mu, \nu}[\mathbb{E}^V-\mathbb{E}^G] \theta [\int_I x(E)q(\operatorname{Tr}f_E(Q_1))dE]=o(1).
\end{equation}
Furthermore, by Taylor expansion, (\ref{bulkremainder1}) and the definition of $\chi$, it suffices to prove
\begin{equation} \label{bulkfinalneedtoprove}
\lim_{N \rightarrow \infty} \max_{\delta N_k \leq l \leq (1-\delta)N_k }  \max_{\mu, \nu}[\mathbb{E}^V-\mathbb{E}^G] \theta [\int_I x(E)q(y(E)+\tilde{y}(E))dE]=o(1),
\end{equation}
where 
\begin{equation}\label{defn_y}
y(E):=\frac{N}{2\pi}\int_{\mathbb{R}^2} i \sigma f^{\prime \prime}_E(e)\chi(\sigma)m_2(e+i\sigma)\mathbf{1}(|\sigma| \geq \tilde{\eta}_d)ded\sigma,
\end{equation}
\begin{equation}\label{defn_tildey}
\tilde{y}(E):=\frac{N}{2\pi}\int_{\mathbb{R}^2} (i f_E(e)\chi^{\prime}(\sigma)-\sigma f^{\prime}_E(e)\chi^{\prime}(\sigma))m_2(e+i\sigma)ded\sigma.
\end{equation}
Next we will use the Green function comparison argument to prove (\ref{bulkfinalneedtoprove}). In the proof of Lemma \ref{greenfunction_comp}, we use the resolvent expansion till the order of 4. However, due to the larger bounds in (\ref{bulkboundsall}),   we will use the following expansion,
\begin{equation} \label{taylorexpansionbulk}
S=R-RVR+(RV)^2R-(RV)^3R+(RV)^4R-(RV)^5S.
\end{equation}
Recall (\ref{srt}) and (\ref{finalgreenfunction}), we have
\begin{equation} \label{reducedbulk1}
[\mathbb{E}^V-\mathbb{E}^G] \theta [\int_I x(E)q(y(E)+\tilde{y}(E))dE]=\sum_{\gamma=1}^{\gamma_{\max}}\mathbb{E}\left( \theta[(\int_I x^S q(y^S+\tilde{y}^S))]-\theta[(\int_I x^T q(y^T+\tilde{y}^T))] \right). 
\end{equation} 
We still use the same notation $\Delta x(E):=x^S(E)-x^R(E).$
We basically follow the approach of Section \ref{greenfunction_comp_edge}, where the control (\ref{edgebound}) is replaced by (\ref{bulkboundsall}). We firstly deal with $x(E).$ 
Denote $\Delta x^{(k)}(E)$ by the summations of the terms in $\Delta x(E)$ containing $k$ numbers of $X^G_{i \mu_1}$. Similar to the discussion of Lemma \ref{lem_random_bound}, recall (\ref{defn_deltax}), by (\ref{INTRODUCTIONEQ}) and (\ref{bulkboundsall}), with $1-N^{-D_1}$ probability, we have
\begin{equation*}
\vert \Delta x^{(5)}(E) \vert \leq N^{-3/2+C\epsilon_0}, \ M+1\leq k \neq  \mu, \nu \leq M+N. 
\end{equation*} 
This yields that
\begin{equation} \label{bulkxde}
\Delta x(E)=\sum_{p=1}^4 \Delta x^{(p)}(E)+O(N^{-3/2+C\epsilon_0}).
\end{equation}
Denote 
\begin{equation*}
\Delta \tilde{y}(E)=\tilde{y}^S(E)-\tilde{y}^R(E), \  \Delta m_2:=m_2^S-m_2^R=\frac{1}{N} \sum_{\mu=M+1}^{M+N}(S_{\mu \mu}-R_{\mu \mu}). 
\end{equation*}
We first deal with (\ref{defn_tildey}). By the definition of $\chi,$ we need to restrict 
$\frac{1}{2} \leq |\sigma| \leq 1$; hence, by (\ref{thm_anisotropic_eq12}), with $1-N^{-D_1}$ probability, we have 
\begin{equation} \label{smallbulkedgebound}
\max_{\mu} |G_{\mu \mu}| \leq N^{\epsilon_1}, \ \max_{\mu \neq \nu} |G_{\mu \nu}| \leq N^{-1/2+\epsilon_1}.
\end{equation}
By (\ref{combinations}), (\ref{combinations1}), (\ref{taylorexpansionbulk}) and (\ref{smallbulkedgebound}), with $1-N^{-D_1}$ probability, we have $|\Delta m_2^{(5)}| \leq N^{-7/2+9\epsilon_1}$.  This yields the following decomposition
\begin{equation}\label{bulkdelyde}
\Delta \tilde{y}(E)=\sum_{p=1}^4 \Delta \tilde{y}^{(p)}(E)+O(N^{-5/2+C\epsilon_0}).
\end{equation}
Next we will control (\ref{defn_y}). Denote $\Delta y(E):=y^S(E)-y^R(E).$  By (\ref{combinations}), (\ref{combinations1}) and (\ref{bulkbound}), with $1-N^{-D_1}$ probability, we have
\begin{equation} \label{ymbound}
|\Delta m_2^{(5)}| \leq N^{-5/2+C\epsilon_0}.
\end{equation}
In order to estimate $\Delta y(E),$ we integrate (\ref{defn_y}) by parts, first in $e$ then in $\sigma$, by (5.24) of \cite{KY}, with $1-N^{-D_1}$ probability, we have
\begin{align} \label{ymintergate}
&\left | \frac{N}{2\pi}\int_{\mathbb{R}^2} i \sigma f^{\prime \prime}_E(e)\chi(\sigma)\Delta^{(5)} m_2(e+i\sigma)\mathbf{1}(|\sigma| \geq \tilde{\eta}_d)ded\sigma \right | \nonumber \\
& \leq CN \left | \int  f^{\prime}_E(e) \tilde{\eta}_d \Delta m_2^{(5)}(e+i\tilde{\eta}_d) de \right |+ CN \left | \int  f^{\prime}_E(e) de \int_{\tilde{\eta}_d}^{\infty}  \chi^{\prime}(\sigma) \sigma \Delta m_2^{(5)(e+i\sigma)}d \sigma \right | \nonumber \\
& +CN \left | \int f^{\prime}_E(e)  de \int_{\tilde{\eta}_d}^{\infty}  \chi(\sigma) \Delta m_2^{(5)}(e+i\sigma) d \sigma \right |.
\end{align}
By (\ref{ymbound}) and (\ref{ymintergate}), with $1-N^{-D_1}$ probability, we have the following decomposition
\begin{equation} \label{bulkdeltayde}
\Delta y(E)=\sum_{p=1}^4 \Delta y^{(p)}(E)+O(N^{-5/2+C\epsilon_0}).
\end{equation}
Similar to the discussion of (\ref{bulkxde}),  (\ref{bulkdelyde}) and (\ref{bulkdeltayde}), it is easy to check that with $1-N^{-D_1}$ probability, we have 
\begin{equation}\label{bulkelebound1}
\int_I |\Delta x^{(p)}(E)|dE \leq N^{-p/2+C \epsilon_0}, \ |\Delta \tilde{y}^{(p)}(E)| \leq N^{-p/2+C \epsilon_0}, \  |\Delta y^{(p)}(E)| \leq N^{-p/2+C \epsilon_0},
\end{equation}
where $p=1,2,3,4$ and $C>0$ is some constant. Furthermore, by (\ref{bulkbound}), with $1-N^{D_1}$ probability, we have
\begin{equation} \label{bulkinterbound}
\int_I |x(E)|dE \leq N^{C\epsilon_0}.
\end{equation}
Due to the similarity of (\ref{bulkdelyde}) and (\ref{bulkdeltayde}), we denote $\ \bar{y}=y+\tilde{y}$ and then we have
\begin{equation} \label{ybardecomposition}
\Delta \bar{y}=\sum_{p=1}^4 \Delta \bar{y}^{(p)}(E)+O(N^{-5/2+C\epsilon_0}).
\end{equation}
By (\ref{bulkelebound1}),  (\ref{ybardecomposition}) and Taylor expansion, we have 
\begin{align} \label{ytaylorbulk}
q(\bar{y}^S)=q(\bar{y}^R)+q^{\prime}(\bar{y}^R)\left( \sum_{p=1}^4 \Delta \bar{y}^{(p)}(E) \right)+ \frac{1}{2}q^{\prime \prime}(\bar{y}^R)\left( \sum_{p=1}^3 \Delta \bar{y}^{(p)}(E) \right)^2+\frac{1}{6}q^{(3)}(\bar{y}^R)\left( \sum_{p=1}^2 \Delta \bar{y}^{(p)}(E) \right)^3 \nonumber \\
+\frac{1}{24}q^{(4)}(\bar{y}^R)\left(  \Delta \bar{y}^{(1)}(E) \right)^4+o(N^{-2}).
\end{align} 
By (\ref{defn_theta1bd}), we have
\begin{align} \label{totaltaylorbulk}
\theta[\int_I x^S q(\bar{y}^S)dE]-\theta[\int_I x^R q(\bar{y}^R)dE]=\sum_{s=1}^4 \frac{1}{s!} \theta^{(s)}(\int_I x^R q(\bar{y}^R)dE)\left[\int_I x^S q(\bar{y}^S)dE-\int_I x^R q(\bar{y}^R)dE\right]^s+o(N^{-2}).
\end{align}
Inserting $x^S=x^R+ \sum_{p=1}^4 \Delta x^{(p)}$ and (\ref{ytaylorbulk}) into (\ref{totaltaylorbulk}), using the partial expectation argument  as in Section \ref{greenfunction_comp_edge}, by (\ref{defn_theta1bd}), (\ref{bulkelebound1}) and (\ref{bulkinterbound}),  we find that that exists a random variable $B$ that depends on the randomness only through $O$ and  the first four moments of $X^{G}_{i \mu_1}$, such that 
\begin{equation}
\mathbb{E} \theta[\int_I x^S q(y+\tilde{y})^S dE]-\mathbb{E} \theta [\int_I x^R q(y+\tilde{y})^R dE]=B+o(N^{-2}).
\end{equation}
Hence, combine with (\ref{reducedbulk1}), we prove (\ref{bulkfinalneedtoprove}), which implies (\ref{bulkcomparisoneq}). This finishes our proof.
\end{proof}
\end{proof}
\begin{appendix}
\section{Proof of Lemma \ref{lem_deltaestimatetileeta}}\label{appendix_a}
In this appendix, we will follow the basic approach of \cite[Lemma 6.1]{EYY} to prove Lemma \ref{lem_deltaestimatetileeta}, which compares the sharp counting function with its delta approximation smoothed on the scale $\tilde{\eta}.$ 
\begin{proof}[Proof of Lemma \ref{lem_deltaestimatetileeta}] Recall (\ref{defn_edomain}), we have $\tilde{\eta} \ll t \ll E_U-E^- \leq \frac{7}{2}N^{-2/3+\epsilon} .$ Furthermore, for $x \in \mathbb{R},$ we have
\begin{equation} \label{differencecharcteristicfunction}
|\mathcal{X}_E(x)-\mathcal{X}_E*\vartheta_{\tilde{\eta}}(x)|=\left | (\int_{\mathbb{R} } \mathcal{X}_E(x)-\int_{E^--x}^{E_U-x})\vartheta_{\tilde{\eta}}(y)dy \right |.
\end{equation}
Denote $d(x):=|x-E^-|+\tilde{\eta}$ and $d_U(x):=|x-E_U|+\tilde{\eta}$, we need  the following bound to estimate (\ref{differencecharcteristicfunction}).
\begin{lem} There exists some constant $C>0,$ such that 
\begin{equation*}
|\mathcal{X}_E(x)-\mathcal{X}_E*\vartheta_{\tilde{\eta}}(x)| \leq C \tilde{\eta}\left[\frac{E_U-E^-}{d_U(x)d(x)}+\frac{\mathcal{X}_E(x)}{(d_U(x)+d(x))}\right].
\end{equation*}
\begin{proof}
When $x>E_U, $ we have
\begin{align*}
|\mathcal{X}_E(x)-\mathcal{X}_E*\vartheta_{\tilde{\eta}}(x)| = \tilde{\eta} \left | \int_{x-E_U}^{x-E^-} \frac{1}{\pi(y^2+\tilde{\eta}^2)} dy\right |&=\frac{\tilde{\eta}}{\pi}\left[\int_{x-E_U}^{x-E^-} \frac{1}{(y+\tilde{\eta})^2}+\frac{2 \tilde{\eta}y}{(y^2+\tilde{\eta}^2)(y+\tilde{\eta})^2} dy\right] \nonumber \\
& \leq C \tilde{\eta} \frac{E_U-E^-}{d_U(x)d(x)}.
\end{align*}
Similarly, we can prove when $x<E^-.$ When $E^- \leq x \leq E_U, $  we have
\begin{equation*}
|\mathcal{X}_E(x)-\mathcal{X}_E*\vartheta_{\tilde{\eta}}(x)| \leq \frac{C\tilde{\eta}}{d_U(x)}+\frac{C \tilde{\eta}}{d(x)}=C \tilde{\eta}\left[\frac{E_U-E^-}{d_U(x)d(x)}+\frac{2\tilde{\eta}}{d_U(x)d(x)}\right],
\end{equation*}
where we use (\ref{anestimationequ}). Therefore, it suffices to show that
\begin{equation}\label{lastlemmastep}
d_U(x)d(x) \geq \frac{1}{4}\tilde{\eta}(d_U(x)+d(x))=\frac{1}{4}\tilde{\eta}(E_U-E^-+2\tilde{\eta}).
\end{equation}
An elementary calculation yields that $d_U(x)d(x) \geq \tilde{\eta}(E_U-E^-+\tilde{\eta}),$ which implies (\ref{lastlemmastep}). Hence, we conclude our proof. 
\end{proof}
\end{lem}

For the right-hand side of (\ref{differencecharcteristicfunction}), when $\min \{d(x), d_U(x)\} \geq t, $ it will be bounded by $O(N^{-3\epsilon_0+\epsilon});$  when $ \min\{d(x), d_U(x)\} \leq t, $ then we must have $\max \{d(x), d_U(x)\} \geq (E_U-E^-)/2$, therefore, it will be bounded by a constant $c$ as $\min \{d(x), d_U(x)\} \geq \tilde{\eta}.$ Therefore, by using the above results for the diagonal elements of $Q_1,$ we have
\begin{align} \label{trdiff1stest}
|\operatorname{Tr}\mathcal{X}_E(Q_1)& -\operatorname{Tr}\mathcal{X}_E*\vartheta_{\tilde{\eta}}(Q_1)| \leq  C\left[\operatorname{Tr}f(Q_1)+ c\mathcal{N}(E^--t, E^-+t)+N^{-3\epsilon_0+\epsilon}\mathcal{N}(E^-+t, E_U-t) \right.  \nonumber \\
&\left. + c\mathcal{N}(E_U-t,E_U+t)+N^{-3\epsilon_0+\epsilon}\mathcal{N}(E_U+t, a_{2k-2})+ \sum_{i=1}^M \mathcal{X}_E*\vartheta_{\tilde{\eta}}((Q_1)_{ii})\mathbf{1}((Q_1)_{ii}>a_{2k-2}) \right],
\end{align}
where $f$ is defined as
\begin{equation*}
f(x):=\frac{\tilde{\eta}(E_U-E^-)}{d_U(x)d(x)}\mathbf{1}(x \leq E^--t).
\end{equation*}
Assume that $\epsilon<\epsilon_1 \epsilon_0,$ by Assumption \ref{defnregularity}, (\ref{thm_rigidity_equ}) and the fact $\epsilon_1<\epsilon$, with $1-N^{-D_1}$ probability, we have
\begin{equation*}
\mathcal{N}(E_U-t, E_U+t)=0, \ \mathcal{N}(E_U+t, a_{2k-2})=0, \   \mathcal{N}(E^-+t, E_U-t) \leq N^{\epsilon_0}.
\end{equation*}
On the other hand, when $(Q_1)_{ii}>a_{2k-2}$, by Assumption \ref{defnregularity}, we have 
\begin{align*}
\mathcal{X}_E*\vartheta_{\tilde{\eta}}((Q_1)_{ii})=\tilde{\eta}\int_{(Q_1)_{ii}-E_U}^{(Q_1)_{ii}-E^-}\frac{1}{y^2+\tilde{\eta}^2}dy \leq \tilde{\eta} \int_{(Q_1)_{ii}-E_U}^{(Q_1)_{ii}-E^-}\frac{1}{y^2}dy \leq \frac{7}{2\tau^2} N^{-4/3+\epsilon-9\epsilon_0},
\end{align*}
where $\tau$ is defined in Assumption \ref{defnregularity}.  Hence, we have $\sum_{i=1}^M \mathcal{X}_E*\vartheta_{\tilde{\eta}}((Q_1)_{ii})\mathbf{1}((Q_1)_{ii}>a_{2k-2})) \leq C N^{-1/3+\epsilon-9\epsilon_0}.$ Therefore, (\ref{trdiff1stest}) can be bounded in the following way
\begin{equation*}
|\operatorname{Tr}\mathcal{X}_E(Q_1)-\operatorname{Tr}\mathcal{X}_E*\vartheta_{\tilde{\eta}}(Q_1)| \leq C(\operatorname{Tr} f(Q_1)+\mathcal{N}(E^--t, E^-+t)+N^{-2\epsilon_0}).
\end{equation*}
To finish our proof, we need to show that with $1-N^{-D_1}$ probability, $\operatorname{Tr}f(Q_1) \leq N^{-2\epsilon_0}.$ By (6.16) of \cite{EYY}, we have
\begin{equation}
\frac{f(x)}{\tilde{\eta}(E_U-E^-)} \leq C(g*\vartheta_{t})(E^--x),
\end{equation}
where $g(y)$ is defined as $g(y):=\frac{1}{y^2+t^2}. $ Recall (\ref{EEEE}) and (\ref{thetaeta}), we have
\begin{equation*}
\frac{1}{N} \operatorname{Tr}\vartheta_t(Q_1-E^-)=\frac{1}{\pi}\operatorname{Im} m_2(E^-+it).
\end{equation*}
Hence, we can obtain that
\begin{align} \label{finallargeequappendix}
\operatorname{Tr} f(Q_1) & \leq  C N \tilde{\eta}(E_U-E^-) \int_{\mathbb{R}} \frac{1}{y^2+t^2} \operatorname{Im}m_2(E^--y+it)dy  \nonumber \\
& \leq CN^{1/3+\epsilon}\tilde{\eta} \int_{\mathbb{R}} \frac{1}{y^2+t^2}[\operatorname{Im}m(E^--y+it)+\frac{N^{\epsilon_1}}{N t}]dy,
\end{align}
where we use (\ref{thm_anisotropic_eq2}). It is easy to check that 
\begin{equation}\label{final0}
CN^{-1/3+\epsilon+\epsilon_1-9\epsilon_0} \int_{\mathbb{R}} \frac{1}{y^2+t^2}\frac{1}{N t}dy \leq CN^{-4/3+\epsilon+\epsilon_1-9\epsilon_0} t^{-2} \int_{\mathbb{R}}\frac{t}{t^2+y^2}dy \leq N^{-2\epsilon_0}.
\end{equation}
Next, we will use (\ref{squareroot}) to estimate (\ref{finallargeequappendix}). When $E^--y \geq a_{2k-1},$ we have
\begin{equation*}
\operatorname{Im}m(E^--y+it) \leq C \sqrt{t+E^--y-a_{2k-1}}. 
\end{equation*}
Denote  $A:=\{E^--y-a_{2k-1} \geq t\}.$ Then we have
\begin{equation} \label{final1}
\int_A \frac{\operatorname{Im}m(E^--y+it)}{y^2+t^2}dy \leq C\int_{\mathbb{R}}\frac{|y|^{1/2}+|E^--a_{2k-1}|^{1/2}}{y^2+t^2}dy \leq C(\frac{1}{t^{1/2}}+\frac{|E^--a_{2k-1}|^{1/2}}{y^2+t^2}),
\end{equation}
\begin{equation}\label{final2}
\int_{A^c} \frac{\operatorname{Im}m(E^--y+it)}{y^2+t^2}dy \leq C t^{-1/2}.
\end{equation}
The other case can be treated similarly. Therefore, by (\ref{finallargeequappendix}), (\ref{final0}), (\ref{final1}) and  (\ref{final2}), we have proved $\operatorname{Tr}f(Q_1) \leq N^{-2\epsilon_0}$ holds true with $1-N^{-D_1}$ probability. Hence, we conclude our proof. 
\end{proof}
\end{appendix}

\vspace{0.5cm}
\noindent{\bf Acknowledgements.} I am very grateful to Jeremy Quastel and B{\' a}lint Vir{\' a}g for many valuable insights and helpful suggestions, which have significantly improved the paper. I would also like to thank my friend Fan Yang for many useful discussions and pointing out some references, especially \cite{XYY}.

\bibliographystyle{abbrv}
\bibliography{eigenvector_bib1}

\begin{thebibliography}{10}

\bibitem{BMP}
Z.~Bai, B.~Miao, and G.~Pan.
\newblock On asymptotics of eigenvectors of large sample covariance matrix.
\newblock {\em Ann. Prob.}, 35:1532--1572, 2007.

\bibitem{BS}
Z.~Bai and J.~Silverstein.
\newblock {\em Spectral Analysis of Large Dimensional Random Matrices}.
\newblock Springer, 2 edition, 2010.

\bibitem{BBP}
J.~Baik, G.~Ben~Arous, and S.~P{\' e}ch{\' e}.
\newblock Phase transition of the largest eigenvalue for nonnull complex sample
  covariance matrices.
\newblock {\em Ann. Prob.}, 33:1643--1697, 2005.

\bibitem{BPZ1}
Z.~Bao, G.~Pan, and W.~Zhou.
\newblock Universality for the largest eigenvalue of sample covariance matrices
  with general population.
\newblock {\em Ann. Stat.}, 43:382--421, 2015.

\bibitem{BGGM}
F.~Benaych-Georges, A.~Guionnet, and M.~Maida.
\newblock Fluctuations of the extreme eigenvalues of finite rank deformations
  of random matrices.
\newblock {\em Electr.J.Prob.}, 16:1621--1662, 2011.

\bibitem{BGN}
F.~Benaych-Georges and R.~Nadakuditi.
\newblock The eigenvalues and eigenvectors of finite, low rank perturbations of
  large random matrices.
\newblock {\em Adv. Math.}, 227:494--521, 2011.

\bibitem{BEKYY}
A.~Bloemendal, L.~Erd{\H o}s, A.~Knowles, H.-T. Yau, and J.~Yin.
\newblock Isotropic local laws for sample covariance and generalized
  \uppercase{W}igner matrices.
\newblock {\em Electr.J. Prob.}, 19:1--53, 2014.

\bibitem{BKYY}
A.~Bloemendal, A.~Knowles, H.-T. Yau, and J.~Yin.
\newblock On the principal components of sample covariance matrices.
\newblock {\em Prob. Theor. Rel. Fields}, 164:459--552, 2016.

\bibitem{PBY}
P.~Bourgade and H.-T. Yau.
\newblock The eigenvector moment flow and local quantum unique ergodicity.
\newblock {\em Commun. Math. Phys.}, 350:231--278, 2017.

\bibitem{CRZ}
T.~Cai, Z.~Ren, and H.~Zhou.
\newblock Estimating structured high-dimensional covariance and precision
  matrices: optimal rates and adaptive estimation.
\newblock {\em Electron. J. Stat.}, 10:1--59, 2016.

\bibitem{CDF}
M.~Capitaine, C.~Donati-Martin, and D.~F{\' e}ral.
\newblock The largest eigenvalues of finite rank deformation of large
  \uppercase{W}igner matrices: Convergence and nonuniversality of the
  fluctuations.
\newblock {\em Ann. Prob.}, 37:1--47, 2009.

\bibitem{DXC20171}
X.~Ding.
\newblock High dimensional deformed rectangular matrices with applications in
  matrix denoising.
\newblock {\em arXiv: 1702.06975}, 2017.

\bibitem{DY}
X.~Ding and F.~Yang.
\newblock A necessary and sufficient condition for edge universality at the
  largest singular values of covariance matrices.
\newblock {\em Ann. Appl. Probab. (to appear)}, 2017.

\bibitem{KN}
N.~El~Karoui.
\newblock Tracy-{W}idom limit for the largest eigenvalue of a large class of
  complex sample covariance matrices.
\newblock {\em Ann. Prob.}, 35:663--714, 2007.

\bibitem{ERSY}
L.~Erd{\H o}s, J.~Ramirez, B.~Schlein, and H.-T. Yau.
\newblock Universality of sine-kernel for {W}igner matrices with a small
  {G}aussian perbubation.
\newblock {\em Electr.J. Prob.}, 15:526--604, 2010.

\bibitem{EYY}
L.~Erd{\H o}s, H.-T. Yau, and J.~Yin.
\newblock Rigidity of eigenvalues of generalized {W}igner matrices.
\newblock {\em Adv. Math.}, 229:1435--1515, 2012.

\bibitem{GL}
G.~Golub and C.~{V}an Loan.
\newblock {\em Matrix Computations}.
\newblock John Hopkins University Press, 4th edition, 2013.

\bibitem{IJ2}
I.~Johnstone.
\newblock On the distribution of the largest eigenvalue in principal components
  analysis.
\newblock {\em Ann. Statist.}, 29:295--327, 2001.

\bibitem{KY}
A.~Knowles and J.~Yin.
\newblock Eigenvector distribution of {W}igner matrices.
\newblock {\em Prob. Theor. Rel. Fields}, 155:543--582, 2013.

\bibitem{KY2}
A.~Knowles and J.~Yin.
\newblock The outliers of a deformed \uppercase{W}igner matrix.
\newblock {\em Ann. Prob.}, 42:1980--2031, 2014.

\bibitem{KY1}
A.~Knowles and J.~Yin.
\newblock Anisotropic local laws for random matrices.
\newblock {\em Prob. Theor. Rel. Fields}, pages 1--96, 2016.

\bibitem{LP}
O.~Ledoit and S.~P{\' e}ch{\' e}.
\newblock Eigenvectors of some large sample covariance matrix ensembles.
\newblock {\em Prob. Theor. Rel. Fields}, 151:233--264, 2011.

\bibitem{LS4}
J.~Lee and K.~Schnelli.
\newblock Tracy-{W}idom distribution for the largest eigenvalue of real sample
  covariance matrices with general population.
\newblock {\em Ann. Appl. Probab.}, 26:3786--3839.

\bibitem{MP}
V.~Mar{\v c}enko and L.~Pastur.
\newblock Distribution of eigenvalues for some sets of random matrices.
\newblock {\em Mathematics of the USSR-Sbornik}, 1:457, 1967.

\bibitem{ORVW}
S.~O'Rourke, V.~Vu, and K.~Wang.
\newblock Eigenvectors of random matrices: a survey.
\newblock {\em Journal of Combinatorial Theory, Series A}, 144:361--442, 2016.

\bibitem{PY}
N.~Pillai and J.~Yin.
\newblock Edge universality of correlation matrices.
\newblock {\em Ann. Stat.}, 40:1737--1763, 2012.

\bibitem{PY1}
N.~Pillai and J.~Yin.
\newblock Universality of covariance matrices.
\newblock {\em Ann. Appl. Probab.}, 24:935--1001, 2014.

\bibitem{JS2}
J.~Silverstein.
\newblock Some limit theorems on the eigenvectors of large-dimensional sample
  covariance matrices.
\newblock {\em J. Multivar. Anal.}, 18:295--324, 1984.

\bibitem{JS1}
J.~Silverstein.
\newblock The {S}tieltjes transform and its role in eigenvalue behavior of
  large dimensional random matrices.
\newblock Random Matrix Theory and its Applications, Lecture Notes Series.
  World Scientific, Singapore, 2009.

\bibitem{SC}
J.~Silverstein and S.~Choi.
\newblock Analysis of the limiting spectral distribution of large dimensional
  random matrices.
\newblock {\em J. Multivar. Anal.}, 54:295--309, 1995.

\bibitem{TV}
T.~Tao and V.~Vu.
\newblock Random matrices: universal properties of eigenvectors.
\newblock {\em Random Matrices Theory Appl.}, 1:1150001, 2012.

\bibitem{TW}
C.~Tracy and H.~Widom.
\newblock On orthogonal and symplectic matrix ensembles.
\newblock {\em Comm. Math. Phys.}, 177:727--754, 1996.

\bibitem{XYY}
H.~Xi, F.~Yang, and J.~Yin.
\newblock Local circular law for the product of a deterministic matrix with a
  random matrix.
\newblock {\em Electr.J. Prob.}, 22:no.60, 77pp, 2017.

\bibitem{YZB}
J.~Yao, Z.~Bai, and S.~Zheng.
\newblock {\em Large Sample Covariance Matrices and High-Dimensional Data
  Analysis}.
\newblock Cambridge University Press, 2015.

\end{thebibliography}

%

\end{document}